\documentclass{article}

\usepackage[preprint]{cpal_2024}

\usepackage[utf8]{inputenc} 
\usepackage[T1]{fontenc}    
\usepackage{hyperref}       

\hypersetup{
	colorlinks,
	citecolor=blue,
	linkcolor=blue,
	urlcolor=blue
}

\usepackage{url}            
\usepackage{booktabs}       
\usepackage{amsfonts}       
\usepackage{nicefrac}       
\usepackage{microtype}      
\usepackage{xcolor}         

\usepackage{times}
\usepackage{epsfig}
\usepackage{graphicx}
\usepackage{tikz}
\usepackage{subfig}

\usepackage[linesnumbered,lined,boxed,commentsnumbered,ruled,vlined]{algorithm2e}

\SetCommentSty{mycommfont}

\usepackage{booktabs}

\usepackage{amsfonts}
\usepackage{bm}
\usepackage{amsmath,amsthm,amssymb,rotating, mathrsfs}

\usepackage[capitalise]{cleveref}
\usepackage{crossreftools}
\crefformat{equation}{(#2#1#3)}

\usepackage{footmisc}

\usepackage{mathtools}
\usepackage{cancel} 

\usepackage{xfrac} 
\usepackage{xparse} 

\DeclareMathAlphabet{\mathpzc}{OT1}{pzc}{m}{it}
\usepackage{enumitem}

\newcommand{\cC}{\mathcal{C}}

\newcommand{\cM}{\mathcal{M}}

\newcommand{\cR}{\mathcal{R}}
\newcommand{\cS}{\mathcal{S}}
\newcommand{\cT}{\mathcal{T}}

\newcommand{\bbR}{\mathbb{R}}
\newcommand{\bbS}{\mathbb{S}}

\newcommand{\bPi}{\bm{\Pi}}
\newcommand{\bA}{\bm{A}}

\newcommand{\bC}{\bm{C}}
\newcommand{\bD}{\bm{D}}
\newcommand{\bE}{\bm{E}}

\newcommand{\bH}{\bm{H}}
\newcommand{\bI}{\bm{I}}

\newcommand{\bR}{\bm{R}}

\newcommand{\bT}{\bm{T}}
\newcommand{\bU}{\bm{U}}
\newcommand{\bV}{\bm{V}}
\newcommand{\bW}{\bm{W}}
\newcommand{\bX}{\bm{X}}
\newcommand{\bY}{\bm{Y}}
\newcommand{\bZ}{\bm{Z}}

\newcommand{\ba}{\bm{a}}

\newcommand{\bc}{\bm{c}}

\newcommand{\bp}{\bm{p}}
\newcommand{\bq}{\bm{q}}

\newcommand{\bs}{\bm{s}}
\newcommand{\bt}{\bm{t}}
\newcommand{\bu}{\bm{u}}
\newcommand{\bv}{\bm{v}}

\newcommand{\bx}{\bm{x}}
\newcommand{\by}{\bm{y}}
\newcommand{\bz}{\bm{z}}
\newcommand{\bxi}{\bm{\xi}}
\newcommand{\bzeta}{\bm{\zeta}}

\NewDocumentCommand{\norm}{mG{2}}{\big\|#1\big\|_{#2}}

\DeclareMathOperator{\diag}{diag}

\DeclareMathOperator{\SO}{SO}
\DeclareMathOperator{\SE}{SE}

\newcommand{\argmin}{\mathop{\rm argmin}}

\NewDocumentCommand{\seqp}{mG{n}}{{#1}_1-\cdots+ {#1}_{#2}}
\NewDocumentCommand{\seqm}{mG{n}}{{#1}_1-\cdots- {#1}_{#2}}

\newcommand{\myparagraph}[1]{\noindent\textbf{#1.}}

\newcommand{\rL}{\widetilde{L}}

\DeclareMathOperator{\rgrad}{\widetilde{\nabla}}
\DeclareMathOperator{\retr}{Retr}

\DeclareMathOperator{\proj}{Proj}

\DeclareMathOperator{\trace}{tr}

\newtheorem{theorem}{Theorem}

\newtheorem{prop}{Proposition}
\newtheorem{corollary}{Corollary}
\newtheorem{lemma}{Lemma}

\theoremstyle{definition}
\newtheorem{definition}{Definition}

\newtheorem{assumption}{Assumption}

\definecolor{cvprblue}{rgb}{0.21,0.49,0.74}
\definecolor{myred}{rgb}{0.82, 0.1, 0.26}
\newtheoremstyle{mystyle}
{}
{}
{}
{}
{\bf \color{myred}}
{.}
{ }
{}

\theoremstyle{mystyle}
\newtheorem{example}{Example}

\theoremstyle{remark}
\newtheorem{remark}{Remark}
\title{Block Coordinate Descent on Smooth Manifolds: Convergence Theory and Twenty-One Examples}

\author{%
  Liangzu Peng \quad \quad \quad Ren\'e Vidal \\ 
  Center for Innovation in Data Engineering and Science (IDEAS) \\ 
  University of Pennsylvania \\  
  \texttt{lpenn@seas.upenn.edu, vidalr@seas.upenn.edu}
}

\begin{document}

\maketitle

\begin{abstract}
    Block coordinate descent is an optimization paradigm that iteratively updates one block of variables at a time, making it quite amenable to big data applications due to its scalability and performance. Its convergence behavior has been extensively studied in the (block-wise) convex case, but it is much less explored in the non-convex case. In this paper we analyze the convergence of block coordinate methods on non-convex sets and derive convergence rates on smooth manifolds under natural or weaker assumptions than prior work. Our analysis applies to many non-convex problems, including ones that seek low-dimensional structures (e.g., maximal coding rate reduction, neural collapse, reverse engineering adversarial attacks, generalized PCA, alternating projection); ones that seek combinatorial structures (homomorphic sensing, regression without correspondences, real phase retrieval, robust point matching); ones that seek geometric structures from visual data (e.g., essential matrix estimation, absolute pose estimation); and ones that seek inliers sparsely hidden in a large number of outliers (e.g., outlier-robust estimation via iteratively-reweighted least-squares). While our convergence theory applies to all these problems, yielding novel corollaries, it also applies to other, perhaps more familiar, problems (e.g., optimal transport, matrix factorization, Burer-Monteiro factorization), recovering previously known results. 
\end{abstract}

\section{Introduction}
In this paper, we consider optimization problems of the form
\begin{equation}\label{eq:obj}
    \min_{\bx \in \cM} F(\bx) \ \ \ \ \ \ \ \ \ \  \ \ \ \ \  \cM:=\cM_1\times \cdots \times \cM_b.
\end{equation}
Throughout the paper, we assume $\cM_i$ is a non-empty subset of $\bbR^{n_i}$ and $F$ is differentiable. Since $\cM$ and $F$ are allowed to be non-convex, \cref{eq:obj} is a rather general formulation. Structural assumptions on $\cM$ (e.g., smoothness as a manifold and closedness) will be imposed later when suitable. The product structure of $\cM$ prompts the algorithmic philosophy of \textit{block coordinate methods}: update one block of variables $\bx_i\in\cM_i$ at a time. This is formally summarized in the listing below (\cref{algo:BEM}).


\begin{algorithm}[h]
    \SetAlgoLined
    \DontPrintSemicolon
    Initialization: $\bx^0=[\bx^{0}_{1}; \dots; \bx^{0}_b]\in \cM$ with $\| \bx^0 \|<\infty$ and each $\bx_i^0\in \cM_i$\\
    For $t\gets 0,1,\dots, T$: \\
        \quad \quad For $i\gets1,\dots,b$: \\
        \quad \quad \quad \quad \quad \quad \quad \quad \quad \quad  Update $\bx_i^{t+1}$ from $\bx^{t+1}_1,\dots, \bx^{t+1}_{i-1}, \bx^{t}_{i+1},\dots, \bx^{t}_b$
	\caption{Block Coordinate Descent on Smooth Manifolds} 
 \label{algo:BEM}
\end{algorithm}
\cref{algo:BEM} does not specify how to update $\bx_i^{t+1}$, while in fact we consider three updating rules in the paper:
\begin{align}
            & \bx_i^{t+1} \in  \argmin_{\bxi\in \cM_i} F( \bx^{t+1}_1,\dots, \bx^{t+1}_{i-1}, \bxi, \bx^{t}_{i+1},\dots, \bx^{t}_b ) \label{eq:BEM-i} \\ 
            & \bx_i^{t+1} \in \argmin_{\bxi\in \cM_i} G_i( \bxi, \bq_i^t ), \ \ \ \textnormal{where\ } \bq_i^t:=[\bx^{t+1}_{1};\dots; \bx^{t+1}_{i-1}; \bx^t_{i}; \bx^t_{i+1} ;\dots;\bx^t_{b}] \label{eq:BMM-i} \\
            & 
            \bx_i^{t+1} \gets \retr_{\bx^t_i}(\bs^t_i), \textnormal{\ \ where $\bs^t_i=-\lambda^t_i \cdot \rgrad_i F(\bx^{t+1}_{1},\dots, \bx^{t+1}_{i-1}, \bx^t_{i}, \bx^t_{i+1} ,\dots,\bx^t_{b})$ } \label{eq:BRGD-i}
\end{align}
Rule \cref{eq:BEM-i} performs \textit{block exact minimization}, minimizing the objective $F$ on a block of variables (\cref{subsection:BEM}). Rule  \cref{eq:BMM-i} is a \textit{block majorization minimization} step, which minimizes some upper bound function $G_i$ of $F$ on block $i$ (\cref{subsection:BMM}). Rule \cref{eq:BRGD-i} is a \textit{block Riemannian gradient descent}, which makes a descent step using the retraction $\retr_{\bx^t_i}(\cdot)$ and Riemannian gradient $\rgrad_i F(\cdot)$ of the objective  with respect to $\bx_i$ (\cref{subsection:BRGD}). Finally, it is possible to blend these rules and still obtain a meaningful descent algorithm (\cref{subsection:BEM+BRGD}).

In \cref{section:examples}, we present some examples that motivate studying \cref{algo:BEM}. In \cref{section:prior-art}, we review prior works on block coordinate descent in the optimization context. \cref{section:preliminary,section:convergence,section:convergence-rate} consists of expositions to our theorems, whose proofs are in the appendix. In \cref{section:concolusion} we conclude the paper.


\section{Examples}\label{section:examples}
Here we contribute by exhibiting a diverse and long list of example problems, either classic or very recent, to which \cref{algo:BEM} applies. These examples are rarely spelled out in the context of block coordinate descent\footnote{Other examples (e.g., matrix factorization, Burer-Monteiro factorization) appear more familiar and in fact have been analyzed in the context of \cref{algo:BEM}, and we will discuss them in the later sections.}, and existing theory is limited in, if not incapable of, asserting the convergence behavior of \cref{algo:BEM} applied to these examples. This motivates us to develop a unifying, general framework for analyzing \cref{algo:BEM}, which eventually generates convergence (rate) guarantees for all of these example problems.

We categorize the examples into four subsections pertaining to four different themes. In \cref{subsection:low-dim} we show examples related to promoting low-dimensional structures from data. In \cref{subsection:HS}, we show examples related to a family of estimation problems with combinatorial structures. In \cref{subsection:GV}, we show examples related to estimating geometric structures from visual data. In \cref{subsection:RE}, we show a class of estimation problems in the presence of outlier data, and these can be viewed also as sparsity-promoting examples (where sparsity arises from the patterns of the inlier residuals).

\subsection{Discovering Low-Dimensional Structures via Block Coordinate Descent}\label{subsection:low-dim}

\begin{example}[Maximal Coding Rate Reduction, MCR$^2$]\label{example:MCR2}
Let $\cM_1\subset \bbR^{d\times n}$ and $\cM_2 \subset \bbR^{n\times k}$ be two compact submanifolds. For each $\bPi\in \cM_2$, denote by $\bPi_i$ the $i$-th column of $\bPi$. With some $\epsilon>0$ and $\bm{1}\in\bbR^n$ denoting the vector of all ones, MCR$^2$ refers to the following objective:\footnote{In practice, $\cM_1$ would be the set of matrices whose columns are unit vectors (or the set of matrices whose Frobenius norm is $1$), and $\cM_2$ would be a manifold defined by quadratic equalities or doubly stochastic matrices. These details are important in practice but irrelevant to our exposition. }
\begin{equation}\label{eq:MCR2}
    \begin{split}
        \min_{\bZ\in \cM_1, \bPi\in \cM_2} \ &-\log\det\Big( \bI_d + \frac{d}{n \epsilon^2} \bZ \bZ^\top \Big) \\ 
    &+ \sum_{i=1}^k \frac{\bPi_i^\top \bm{1}}{n} \log \det \Big( \bI_d + \frac{d}{\epsilon^2\cdot \bPi_i^\top \bm{1}} \bZ \diag(\bPi_i) \bZ^\top \Big) 
    \end{split} \tag{MCR$^2$}
\end{equation}
\textbf{\cref{eq:MCR2} pursues low-dimensional structures} as it was shown in \cite{Chan-JMLR2022} that, under some conditions, any global minimizer of \cref{eq:MCR2} exhibits certain (\textit{within class}) low-rank properties. Prior works combined \cref{eq:MCR2} with deep networks 
\cite{Yu-NeurIPS2020,Chan-JMLR2022,Li-arXiv2022,Ding-ICCV2023}, and here we consider \cref{eq:MCR2} through a pure optimization lens. While exact minimization with respect to either $\bPi$ or $\bZ$ is difficult, the Riemannian gradient descent option \cref{eq:BRGD-i} is implementable, making \cref{algo:BEM} applicable. An important special case is when $\bPi$ is given. In this situation, \cref{eq:MCR2} can be reformulated so that it is again amenable to alternating minimization \cite{Baek-CVPR2022}. The reformulated objective is shown in \cite[Eq. 7]{Baek-CVPR2022}, where the two variables $\bm{\Gamma}$ and $\bA$ lie respectively in two compact submanifolds of the Euclidean space. 
Moreover, again, \cref{algo:BEM} can be applied. Therefore, our analysis of \cref{algo:BEM} applied to the general formulation \cref{eq:obj} will generate novel guarantees for \eqref{eq:MCR2} or its reformulation in \cite{Baek-CVPR2022}.



\end{example}

\begin{example}[Neural Collapse]\label{example:NC}
    In a recent line of research termed \textit{neural collapse} \cite{Papyan-PNAS2020,Lu-arXiv2020v2,Fang-PNAS2021,Zhu-NeurIPS2021,Ji-ICLR2022,Han-ICLR2022,Zhou-ICML2022,Yaras-NeurIPS2022,Zhou-NeurIPS2022}, it has been of interest to study the problem of the form (see, e.g., \cite[Eq. 4]{Yaras-NeurIPS2022} and \cite[Theorem 1]{Lu-arXiv2020v2})
    \begin{equation}\label{eq:NC}
        \left( \widehat{\bW}, \widehat{\bH} \right)\in \argmin_{\bW,\bH} F(\bW^\top\bH) \ \ \ \ \ \ \ \textnormal{s.t.} \ \ \ \ \ \ \ \bW\in \cM_1, \bH\in \cM_2, \tag{NC}
    \end{equation}
    where $\cM_1$ and $\cM_2$ are compact submanifolds of the Euclidean space, and $F$ is \textit{block-wise Lipschitz smooth} (i.e., the gradients of $F$ with respect to each block is Lipschitz continuous).\footnote{In \cite{Yaras-NeurIPS2022}, $\cM_1$ and $\cM_2$ are both the set of matrices whose columns are unit vectors and $F$ is related to the soft-max cross-entropy loss and is indeed block-wise Lipschitz smooth; other losses discussed in \cite{Zhou-NeurIPS2022}, such as \textit{mean squared error}, is also block-wise Lipschitz smooth. A slightly more general situation considered in \cite{Fang-PNAS2021} is where $\cM_1$ and $\cM_2$ are defined by some Euclidean balls. To summarize, \cref{eq:NC} generalizes their settings. }
    
    \textbf{\cref{eq:NC} pursues low-dimensional structures} as it was shown in prior works that, under certain conditions, any global minimizer $\widehat{\bH}$ of \eqref{eq:NC} exhibits certain (\textit{within class}) low-rank properties called \textit{neural collapse} (see, e.g., \cite[Theorem 1]{Yaras-NeurIPS2022}). In some sense, \cref{eq:NC} can be viewed as a specific case of the matrix factorization model (see, e.g., \cite[Eq. 1]{Haeffele-TPAMI2019} and \cite[Eq. 1.3]{Olikier-arXiv2023}).  \cref{algo:BEM} can be used to minimize \cref{eq:NC} in an alternating fashion. We will see that our theorems of \cref{algo:BEM} are applicable to \cref{eq:NC}, guaranteeing the convergence to a stationary point; and combining this with \cite[Theorem 2]{Yaras-NeurIPS2022} yields guarantees of finding global minimizers.
\end{example}
\begin{example}[Reverse Engineering $\ell_p$ Adversarial Attacks]\label{example:REAA}
    In image classification, \textit{adversarial attacks} refer to devising imperceptible perturbations added to an input image so that the classification result of a given algorithm is made entirely wrong \cite{Pal-NeurIPS2020,Pal-TMLR2023}. Typically, adversarial attacks involve some additive noise with bounded $\ell_p$ norm, where $p$ could be equal to $1$, $2$, or $\infty$, and its values are determined by the attacker. To design classification methods that are robust to such attacks, it is important to estimate the precise values of $p$. This estimation task is called \textit{reverse engineering $\ell_p$ adversarial attacks}, and we refer the reader to \cite{Thaker-ICML2022} for background and motivations on that topic. What interests us in that context is the optimization formulation shown in \cite[Eq. 11]{Thaker-ICML2022} and proposed to reverse engineer the types of adversarial attacks:
    \begin{equation}\label{eq:REAA}
        \min_{\bc_s, \bc_a} \| \bx - \bD_s \bc_s - \bD_a \bc_a \|^2 + R_s(\bc_s) + R_a(\bc_a)
    \end{equation}
    While the first term $\| \cdot \|^2$ in  \cref{eq:REAA} asks for data fidelity, \textbf{the regularizers $R_s(\cdot)$ and $R_a(\cdot)$ promote low-dimensional structures} as they are $\ell_1$-style losses often used to enforce sparsity; see \cite[Eq. 11]{Thaker-ICML2022} for concrete forms of $R_s(\cdot)$ and $R_a(\cdot)$. Then, block coordinate descent can be applied to update $\bc_s$ and $\bc_a$ in an alternating manner, While \cref{eq:REAA} is actually a convex and unconstrained problem (in the context of \cite{Thaker-ICML2022} or provided that $R_s(\cdot)$ and $R_a(\cdot)$ are convex), we will show how to obtain valid convergence guarantees for \cref{algo:BEM} applied to a more general case where the optimization variables $\bc_s$ and $\bc_a$ have non-convex or manifold constraints.
    
\end{example}
\begin{example}[Generalized Principal Component Analysis, GPCA]
    The classic GPCA approach \cite{Vidal-2003thesis,Vidal-TPAMI2005,Vidal-GPCA2016,Tsakiris-SIAM-J-IS2017} is related to the also classic \textit{subspace clustering} problem \cite{Ma-SIAM-Review2008,Vidal-SPM2011}, with the goal of \textbf{discovering low-dimensional structures (e.g., a union of low-dimensional subspaces) from data}. Specifically, we are given a set of $m$ points $\{ \bp_j \}_{j=1}^m$ lying in a union of $b$ subspaces $\cup_{i=1}^b \cS_i\subset \bbR^{D}$ and we need to recover each subspace $\cS_i$ from data $\{ \bp_j \}_{j=1}^m$. Note that $\cS_i$ is identified uniquely with its orthogonal complement $\cS^\perp_i$, and therefore uniquely (up to isometry) with an orthonormal basis $\bA_i^*$ of $\cS^\perp_i$. Therefore, we will focus on recovering $\bA_i^*$ next. 

To simplify exposition, we assume $b=2$, i.e., there are only two subspaces $\cS_1$ and $\cS_2$ to recover. This is without loss of generality, and a parallel discussion for arbitrary $b$ can be found in the appendix.

A point $\bp$ belongs to $\cS_1\cup \cS_2$ if and only if $\| \bp^\top \bA_1^* \| \cdot \| \bp^\top \bA_2^* \| =0$. Then, given points   $\{ \bp_j \}_{j=1}^m$, we can consider the following \textit{least-squares} formulation of GPCA (cf. \cite[Eq. 3.87]{Vidal-2003thesis}):
\begin{equation}\label{eq:GPCA-LS}
    \min_{ \bA_1,\bA_2 } \sum_{j=1}^m \| \bp_j^\top \bA_1 \|^2 \cdot \| \bp_j^\top \bA_2 \|^2 \ \ \ \ \ \textnormal{s.t.} \ \ \ \   \bA_1^\top \bA_1  = \bI_{c_1}, \ \bA_2^\top \bA_2  = \bI_{c_2}. \tag{GPCA-LS}
\end{equation}
Here, $c_i$ denotes the dimension of $\cS_i^\perp$ and $\bI_{c_i}$ denotes the $c_i\times c_i$ identity matrix ($i=1,2$). 

We find that \cref{eq:GPCA-LS} is amenable to block exact minimization (\cref{algo:BEM}). Indeed, we can start with initialization $\bA_1^0$ and $\bA_2^0$, and contextualize the block minimization step \cref{eq:BEM-i} as
\begin{equation}\label{eq:BEM-i-GPCA}
    \bA_i^{t+1} \in \argmin_{\bxi} \sum_{j=1}^m w_{i,j}^t\cdot  \| \bp_j^\top \bxi \|^2  \ \ \ \ \ \textnormal{s.t.} \ \ \ \   \bxi^\top \bxi  = \bI_{c_i}  \ \ \ (i=1,2)
\end{equation}
where the weights $w_{1,j}^t$ and $w_{2,j}^t$ are defined to be
\begin{equation}\label{eq:GPCA-LS-weight}
    w_{1,j}^t:= \| \bp_j^\top \bA_{2}^t \|^2, \ \ \ \ w_{2,j}^t  := \| \bp_j^\top \bA_{1}^{t+1} \|^2.
\end{equation}
One can solve \cref{eq:BEM-i-GPCA} easily: The iterate $\bA_i^{t+1}$ of \cref{eq:BEM-i-GPCA} is given as the matrix whose columns are eigenvectors of the weighted covariance matrix $\sum_{j=1}^m w_{i,j}^t\cdot \bp_j \bp_j^\top$ corresponding to its smallest $c_i$  eigenvalues.

The above setup is ideal in the sense that \textit{all} points $\bp_j$'s lie in the subspace union $\cS_1\cup \cS_2$. In a more challenging situation, there could be \textit{some} ``\textit{outlier}'' points $\bp_j$ that are \textit{far} from $\cS_1\cup \cS_2$. Since formulation \cref{eq:GPCA-LS} is of the least-squares type, it can be sensitive to such outliers. To alleviate the influence of outliers, a typical choice is to consider the following Huber-style formulation:
\begin{equation}\label{eq:GPCA-Huber}
    \min_{ \bA_1,\bA_2 } \sum_{j=1}^m h\big(\| \bp_j^\top \bA_1 \| \big) \cdot h\big( \| \bp_j^\top \bA_2 \| \big) \ \ \ \ \ \textnormal{s.t.} \ \ \ \   \bA_1^\top \bA_1  = \bI_{c_1}, \ \bA_2^\top \bA_2  = \bI_{c_2} \tag{GPCA-Huber}
\end{equation}
Here $h(\cdot)$ is a (rescaled and translated) Huber loss defined as (with some parameter $\epsilon>0$)
\begin{equation}\label{eq:Huber}
    h(r):= \begin{cases}
        |r| & |r|> \epsilon \\
       \frac{r^2 + \epsilon^2 }{2\epsilon} & |r| \leq \epsilon
    \end{cases} \tag{Huber}
\end{equation}
Again, \cref{algo:BEM} can be applied to \cref{eq:GPCA-Huber}. If one  insists on exact minimization \cref{eq:BEM-i}, then one would attempt to solve, for $i=1,2$, the optimization problem
\begin{equation}\label{eq:BEM-i-GPCA-Huber}
    \bA_i^{t+1} \in \argmin_{\bxi} \sum_{j=1}^m v_{i,j}^t\cdot  h(\| \bp_j^\top \bxi \|)  \ \ \ \ \ \textnormal{s.t.} \ \ \ \   \bxi^\top \bxi  = \bI_{c_i},
\end{equation}
where weights $v_{1,j}^t$ and $v_{2,j}^t$ are defined as $v_{1,j}^t:= h\big( \| \bp_j^\top \bA_{2}^t \| \big)$ and $ v_{2,j}^t  := h\big( \| \bp_j^\top \bA_{1}^{t+1} \| \big)$ respectively, similarly to \cref{eq:GPCA-LS-weight}. While \cref{eq:BEM-i-GPCA-Huber} does not have a closed-form solution, we will show how the majorization minimization option \cref{eq:BMM-i} can address this issue in a more general context in \cref{subsection:BMM}.
\end{example}

\begin{example}[Alternating Projection]\label{example:alt-proj}
    In \cite{Zhu-arXiv2018b} the following problem was considered:   
    \begin{equation}\label{eq:minimum-dist}
        \min_{\bx_1\in \cM_1,\bx_2\in \cM_2} F(\bx_1,\bx_2), \ \ \ \ \ F(\bx_1,\bx_2):=\|\bx_1 - \bx_2 \|^2 \tag{AP}
    \end{equation}
    \textbf{\cref{eq:minimum-dist} pursues low-dimensional structures} as it aims to compute a point in the intersection $\cM_1\cap \cM_2$ of two sets or low-dimensional manifolds (assuming $\cM_1\cap \cM_2 \neq \varnothing$). \cref{algo:BEM} with rule \cref{eq:BEM-i} is applicable to \cref{eq:minimum-dist}, and in fact the convergence and convergence rate of this method was studied in \cite{Zhu-arXiv2018b}. However, their proof involves non-smooth analysis techniques and assumes the so-called \textit{Kurdyka-Lojasiewicz property}, which might be hard to verify. In this paper, we will show how to obtain more informative rates via simpler analysis.
\end{example}

\subsection{Discovering Combinatorial Structures via Block Coordinate Descent}\label{subsection:HS}
\begin{example}[Homomorphic Sensing]\label{example:HS}
    The \textit{homomorphic sensing} problem \cite{Tsakiris-ICML2019,Peng-ICML2021,Peng-ACHA2021} is related to solving
    \begin{equation}\label{eq:HS}
        \min_{\bx\in\bbR^n,\bT\in \cT} \|\by - \bT \bA \bx \|^2,  \tag{HS}
    \end{equation}
    Here, $\cT$ is assumed to be a finite set of matrices, adding some combinatorial flavor to \cref{eq:HS}. More concretely, $\cT$ can be the set of permutation matrices, and then \cref{eq:HS} models the problem of \textit{linear regression without correspondences} \cite{Unnikrishnan-Allerton2015,Unnikrishnan-TIT18,Pananjady-TIT18,Slawski-JoS19,Tsakiris-TIT2020,Peng-SPL2020,Slawski-JCGS2021,Onaran-arXiv2022,Mazumder-MP2023}; or $\cT$ can be the set of \textit{sign matrices}, that is the set of diagonal matrices whose diagonal elements are either $1$ or $-1$, and then \cref{eq:HS} corresponds to the problem of \textit{real phase retrieval} \cite{Balan-ACHA2006} or \textit{symmetric mixture of two linear regressions} \cite{Balakrishnan-AoS2017,Klusowski-TIT2019}. In both cases, $\cT$ is a compact submanifold of the Euclidean space.  See \cite[Section 4]{Peng-ICML2021} and \cite[Section 1.2]{Peng-ACHA2021} for more relevant examples and discussions.
    
      
 To solve \cref{eq:HS}, it is natural to consider minimizing $\bx$ and $\bT$ in an alternating style, an instance of \cref{algo:BEM} combined with block exact minimization \cref{eq:BEM-i}. Such alternating minimization algorithm was proposed respectively in \cite{Abid-Allerton2018} with no convergence guarantees for the case of linear regression without correspondences and in \cite{Netrapalli-NeurIPS2013} with convergence proofs for the case of (real) phase retrieval. We will show how to prove general theorems that are applicable to both cases, yielding novel guarantees for the former and different guarantees than \cite{Netrapalli-NeurIPS2013} for the latter. Our theory will furthermore apply to the situation where $\cT$ is the set of matrices whose elements are compositions of permutation and sign matrices (see \cite{Lv-ACM2018} and \cite[Section 4]{Peng-ICML2021}).
\end{example}

\begin{example}[Sparse Homomorphic Sensing]
    One could combine \cref{eq:HS} with a sparsity-promoting regularizer $R(\bx)$, and arrive at the following problem (see \cite{Peng-ICML2021} for a related formulation):
    \begin{equation}\label{eq:SHS}
        \min_{\bx\in\bbR^n,\bT\in \cT} \|\by - \bT \bA \bx \|^2 + R(\bx)  \tag{SHS}
    \end{equation}
    Again, the idea of block coordinate descent is still valid, and our theory will apply.
\end{example}

\begin{example}[Robust Point Matching, RPM] The classic RPM problem \cite{Chui-CVIU03} aims to align two point sets $\{\by_1,\dots,\by_m\}$ and $\{ \bx_1,\dots,\bx_n \}$ via some affine transformation $\bA$. The RPM problem is formulated as
    \begin{equation}\label{eq:RPM}
        \min_{ w_{ij}\in\{0, 1 \} , \bA } \sum_{i=1}^m \sum_{j=1}^n w_{ij} \cdot \| \by_i - \bA \bx_j  \|^2 \quad \textnormal{s.t.} \quad \sum_{i=1}^m w_{ij} \leq 1 (\forall j), \quad  \sum_{i=1}^m \sum_{j=1}^n w_{ij} = k \tag{RPM}
    \end{equation}
Here $\{w_{ij}\}_{ij}$ forms a binary matrix that encodes the correspondences between $\by_i$ and $\bx_j$. The constraint $\sum_{i=1}^m w_{ij} \leq 1$ ($\forall j$) ensures there is only one $\by_i$ corresponding to a given $\bx_j$, and constraint $\sum_{i=1}^m \sum_{j=1}^n w_{ij} = k$ ensures there are (at most) $k$ correspondences. Given $w_{ij}$, \cref{eq:RPM} becomes a least-squares problem, while given $\bA$, \cref{eq:RPM} becomes an instance of the $k$-cardinality assignment problem \cite{Dell-DAM1997}. Both subproblems can be solved in polynomial time, so \cref{algo:BEM} with block exact minimization applies, and so will our theory.
\end{example}

    

\subsection{Discovering Geometric Structures via Block Coordinate Descent}\label{subsection:GV}
In the field called \textit{geometric computer vision}, one aims to discover \textit{geometric} structures (e.g., rotation and translation between two cameras), from \textit{visual} data (e.g., images and point clouds); see, e.g., textbooks \cite{Hartley-2004,Ma-book2004}. Given below are a few examples from geometric vision to which block coordinate descent applies. 
\begin{example}[Essential Matrix Estimation]\label{example:EME}
    Denote by $\SO(3)$ the set of $3\times 3$ rotation matrices, that is $\SO(3):=\{\bR\in \bbR^{3\times 3}: \bR^\top \bR = \bI_3, \det(\bR)=1  \}$. Denote by $\bbS^2$ the unit sphere of $\bbR^3$, that is $\bbS^2:=\{ \bt\in \bbR^3: \|\bt \|=1 \}$. With $\bt= [t_1;t_2;t_3]\in \bbR^3$, denote by $[\bt]_\times$ the $3\times 3$ matrix representing cross product, that is
    \begin{equation*}
        [\bt]_\times := \begin{bmatrix}
    0 & -t_3 & t_2 \\ t_3 & 0 & -t_1 \\ -t_2 & t_1 & 0 
\end{bmatrix}.
    \end{equation*}
    An essential matrix $\bE$ is defined as a $3\times 3$ matrix of the form $[\bt]_{\times} \bR$, where $\bt$ is a unit vector and $\bR$ a rotation. Then \textit{essential matrix estimation} can be formulated as the following optimization problem: 
    \begin{equation}\label{eq:EME}
        \min_{ \bE\in \cM }  \sum_{i=1}^m \left(\bx_i^\top \bE \by_i \right)^2 \quad \textnormal{s.t.} \quad \cM:= \left\{ [\bt]_{\times} \bR:  \bt\in \bbS^2, \bR\in \SO(3)   \right\}   \tag{E}
    \end{equation}
    While in \cite{Zhao-TPAMI2020} it was shown that \cref{eq:EME} can be relaxed into a convex semidefinite program, one can tackle \cref{eq:EME} via block coordinate descent as well, that is to alternate between minimizing \eqref{eq:EME} in variables $\bR$ and $\bt$.
\end{example}
\begin{example}[Generalized Essential Matrix Estimation]
    A generalization of \cref{example:EME} was first considered in \cite{Pless-CVPR2003}, which motivates the following formulation (see also \cite{Li-CVPR2008,Campos-ICRA2019}):
    \begin{equation}\label{eq:GEME}
            \min_{ \bE\in \cM, \bR \in \SO(3) }  \sum_{i=1}^m \left(\bx_i^\top \begin{bmatrix}
                \bE & \bR \\ 
                \bR & 0
            \end{bmatrix} \by_i \right)^2 \quad \textnormal{s.t.} \quad \cM:= \left\{ [\bt]_{\times} \bR:  \bt\in \bbS^2, \bR\in \SO(3)   \right\}   \tag{GE}
    \end{equation}
    Similarly to \cref{example:EME}, and as observed in \cite{Campos-ICRA2019}, \cref{algo:BEM} applies, and it also applies to other geometric vision problems (e.g., \textit{absolute pose estimation}); see \cite{Campos-ICRA2019} for details. We will see our theorems apply, too.
\end{example}

\subsection{Discovering Inliers from Outlier Data via Block Coordinate Descent}\label{subsection:RE}

\begin{example}[IRLS for Outlier-Robust Estimation]\label{example:IRLS}
    Consider the following pair of problems:

\begin{minipage}{.49\textwidth}
  \begin{equation}\label{eq:robust}
      \min_{\bx \in \cC }  \sum_{i=1}^m \rho \big(r_i(\bx) \big) 
  \end{equation}
\end{minipage}%
\begin{minipage}{.49\textwidth}
  \begin{equation}\label{eq:weight+x}
      \min_{\bx \in \cC, w_i\in[0,1] }  \sum_{i=1}^m \Big( w_i\cdot r_i^2(\bx) + R_\rho(w_i) \Big)
  \end{equation}
\end{minipage}

In \cref{eq:robust}, the constraint set $\cC$ is an embedded manifold of the Euclidean space, $r_i(\cdot)$ is some residual function, and $\rho:\bbR\to\bbR$ is some robust loss to alleviate the influence of outliers. 
As shown in \cite[Section 5]{Black-IJCV1996}, under certain conditions on the outlier-robust loss $\rho$, \cref{eq:robust} reduces to problem \cref{eq:weight+x}, which further contains some weight variables $w_i$'s and a regularization term $R_\rho(w_i)$ on each weight $w_i$; the exact expression of $R_\rho(w_i)$ depends on $\rho$. Note that \cref{eq:robust} can be specialized in many ways, as there are more than \textbf{forty} different robust losses  $\rho(\cdot)$ and more than \textbf{ten} different residuals  $r_i(\cdot)$ occurring in practice; see, for example, \cite[Appendix C]{De-CCE2021}, \cite[Table 2]{Chatterjee-TPAMI2017}, \cite[Appendix B]{Peng-CVPR2023}, and \cite{Black-IJCV1996,Zach-BMVC2017,Rahimi-arXiv2022v2}.

Note that \cref{eq:weight+x} is amenable to alternating minimization: Minimizing \cref{eq:weight+x} in $\bx$ is a (constrained) weighted least-squares problem as $R_\rho(w_i)$ does not depend on $\bx$, and minimizing \cref{eq:weight+x} in $w_i$ typically has a closed-form solution. This is essentially the classic IRLS algorithm in its very general form (see, e.g.,  \cite{Coleman-TOMS1980,Ochs-SIAM-J-IS2015,Aftab-WCACV2015}). 
\end{example}

Outlier-robust estimation and iteratively reweighted least-squares (IRLS) are important subjects with a vast literature. Our purpose here is not to provide a complete picture of the two subjects. Instead, we proceed by presenting two examples of $\rho(\cdot)$ and one example of $r_i(\cdot)$ that will be useful for the sequel.

\begin{example}[Geman-McClure Loss \cite{Geman-1985}]\label{example:GM}
    For the Geman-McClure loss
    \begin{equation}
        \rho(r) = \frac{r^2}{1+r^2}, \tag{GM}\label{eq:GM}
    \end{equation}
    the regularizer $R_\rho(w_i)=(\sqrt{w_i} - 1)^2$ makes the two problems \cref{eq:robust} and \cref{eq:weight+x} equivalent in the sense of \cite{Black-IJCV1996} (see \cite[Section 5.2.1 and Figure 28]{Black-IJCV1996}). By checking the second-order derivative, one shows $R_\rho(\cdot)$, and therefore \cref{eq:weight+x}, is strictly convex in $w_i$. Hence, given $\bx$, minimizing \cref{eq:weight+x} in variable $w_i$ can be solved in closed form: Simply check the objective values at the two endpoints $0,1$ and the stationary point where the derivative of the objective with respect to $w_i$ is $0$. In fact, this stationary point is of the form $\frac{1}{(r^2_i(\bx) + 1)^2}$, always lying in $(0,1]$, 
    so it is indeed the global minimizer of \cref{eq:weight+x} given $\bx$. This also suggests the constraint $w_i\in [0,1]$ plays no algorithmic role for IRLS to be applied, but $[0,1]$ is a compact subset of $\bbR$, which will benefit analysis.
\end{example}
Next, we give an example where IRLS minimizes an objective slightly different than that of \cref{eq:weight+x}:
\begin{example}[Smoothed $\ell_1$ Loss]\label{example:smooth-L1}
    With some parameter $\epsilon>0$, consider the following formulation:
    \begin{equation}\label{eq:specific-IRLS}
        \min_{\bx \in \cC, w_i\in (0,\infty) }  \sum_{i=1}^m \Big( w_i\cdot r_i^2(\bx) + \epsilon^2 \cdot w_i + \frac{1}{w_i} \Big)
    \end{equation}
    This formulation was considered in \cite{Daubechies-CPAM2010} in the context of compressed sensing and also analyzed in \cite{Beck-SIAMOpt2015} in the convex case with specific residual functions $r_i(\cdot)$. Similarly to \cref{example:GM}, the objective of \cref{eq:specific-IRLS}  is also strictly convex in $w_i$, therefore the optimal weight given $\bx$ is
    \begin{equation}\label{eq:specific-IRLS-weight}
        w_i = \frac{1}{\sqrt{r_i^2(\bx) + \epsilon^2}}.
    \end{equation}
    and this is how the IRLS algorithm updates weight $w_i$. Finally, substitute $w_i$ of \cref{eq:specific-IRLS-weight} into \cref{eq:specific-IRLS}, and we obtain
    \begin{equation}
        \min_{\bx \in \cC }  \sum_{i=1}^m \sqrt{r_i^2(\bx) + \epsilon^2}. 
    \end{equation}
    This is equivalent to the outlier-robust estimation problem \cref{eq:robust} with loss $\rho$ set to be $\rho(r):= \sqrt{r^2 + \epsilon^2 }$. (Smoothed losses are also considered in matrix rank minimization problems, see, e.g., \cite{Kramer-arXiv2021,Kummerle-ICML2021}.)
\end{example}
\begin{example}[Point Cloud Registration]\label{example:PCR}
    We now specialize \cref{eq:robust} and \cref{eq:weight+x} for the problem of point cloud registration \cite{Horn-JOSAA1988}. In this problem, the residual function $r_i(\cdot)$ involves two variables: a 3D rotation matrix $\bR\in \SO(3)$ and some translation vector $\bt\in\bbR^3$; therefore the constraint set $\cC$ of $(\bR,\bt)$ is the special Euclidean group $\SE(3):= \SO(3)\times \bbR^3$. The precise expression of the residual function $r_i(\cdot)$ is given as
    \begin{equation}\label{eq:PCR-residual}
        r_i(\bR, \bt) = \|\bm{b}_i - \bR \ba_i - \bt \|_2,
    \end{equation}
    where $\ba_i\in\bbR^3$ and $\bm{b}_i\in\bbR^3$ are 3D points. 
\end{example}

\begin{remark}[General Convergence Theory of IRLS for Outlier-Robust Estimation]
    The IRLS paradigm has been widely applied with massive empirical success to many vision, robotics, and machine learning problems formulated in the form of \cref{eq:robust} or \cref{eq:weight+x} \cite{Aftab-TPAMI2014,Zhou-ECCV2016,Tsakiris-JMLR2018,Chung-SIAM-J-SC2019,Dong-PRL2019,Ding-CVPR2020,Shi-ICML2020,Shi-NeurIPS2020,Iwata-ECCV2020,Zhao-TPAMI2020,Yang-RA-L2020,Sidhartha-3DV2021,Zhou-TGRS2021,Antonante-TRO2021,Le-3DV2021,Gazzola-SIAM-J-SC2021,Torbjorn-MIC2022,Mankovich-CVPR2022}. However, existing convergence  analysis of IRLS has either of the following limitations: (i) it is application-specific \cite{Daubechies-CPAM2010,Mohan-JMLR2012,Lai-SIAM-J-NA2013,Lerman-FoCM2015,Lerman-J-IMA2018,Kummerle-JMLR2018,Kummerle-ICML2021,Kummerle-NeurIPS2021,Mukhoty-AISTATS2019,Peng-NeurIPS2022,Kummerle-NeurIPS2023}; (ii) it is for the convex case \cite{Razaviyayn-SIAM-J-Opt2013,Beck-SIAMOpt2015,Hong-MP2017}; (iii) it only reveals asymptotic convergence \cite{Aftab-WCACV2015}. The global rates of IRLS for the general non-convex formulation \cref{eq:robust} or \cref{eq:weight+x} have not been proved for decades. In this paper, we obtain global sublinear rates of IRLS from our general convergence theorems for \cref{algo:BEM}.
\end{remark}



\section{Prior Art on Block Coordinate Descent: Cross the Watershed}\label{section:prior-art}
Would there be existing theorems on block coordinate descent that are directly applicable to the examples of \cref{section:examples}? Indeed, research on problems of the form \cref{eq:obj} and \cref{algo:BEM} has a long history and rich contemporary developments. While this fascinating line of work is too vast to detail, one might venture to assert that researchers are crossing the watershed from \textit{convexity} to \textit{non-convexity} (in Rockafellar's tone \cite{Rockafellar-SIAM-Rev1993}), and from \textit{convergence} analysis to \textit{convergence rate} analysis. That being said, works on the non-convex side are relatively scarce, and existing convergence theory of \cref{algo:BEM} is limited in many ways. To support this assertion and trace the historical trajectory of theoretical advances, we next review prior works.

\myparagraph{Convexity, Convergence} \cref{algo:BEM} made its early appearance\footnote{In this section we say ``\cref{algo:BEM}'' to mean one of its special cases (e.g., one case is to perform exact minimization \cref{eq:BEM-i} for every block). Besides \cref{eq:BEM-i}-\cref{eq:BRGD-i}, there are other update rules for \cref{algo:BEM}; see, e.g., \cite{Wright-MP2015,Shi-arXiv2016v2} for surveys.} in \textit{\textbf{1957}} \cite{Hildreth-1957} to minimize a convex quadratic objective with non-negative variables, and then in \textit{\textbf{1959}} \cite{D-1959} it was proved to converge for a more general situation where the objective is convex  (among other assumptions). In \textit{\textbf{1963}} \cite{Warga-J-SIAM1963}, Warga realized that \cref{algo:BEM} could fail to converge when $F$ is not differentiable (see also Powell's counterexample \cite{Powell-MP1973}), while convergence is guaranteed if $F$ is continuously differentiable and block-wise strictly convex and the $\cM_i$'s are compact and convex (see also \cite[Proposition 3.9]{Bertsekas-1997}). In their \textit{\textbf{1970}} book (\cite{Ortega-1970},  \cite[Theorem 14.6.7]{Ortega-2000}), Ortega \& Rheinboldt proved that \cref{algo:BEM} converges in the unconstrained case, assuming $F$ is strongly convex and differentiable. In the same year \cite{Zadeh-MS1970}, Zadeh showed that, in the unconstrained case with $\cM_i=\bbR$, a \textit{pseudo-convexity} assumption would suffice for convergence (instead of strong convexity). It was not until the turn of the century that Grippo \& Sciandrone extended Zadeh's work, respectively, for the unconstrained case \cite{Grippo-OMS1999} or convex constraints \cite{Grippo-ORL2000}, proving that \cref{algo:BEM} converges as long as $F$ is continuously differentiable and either pseudo-convex \cite[Proposition 5]{Grippo-ORL2000} or block-wise strictly quasi-convex \cite[Proposition 6]{Grippo-ORL2000}. In \textit{\textbf{2001}} \cite{Tseng-JOTA2001}, Tseng unified and extended the above results in two ways. First, he adopted a more general notion of differentiability \cite[Section 3 \& Theorem 4.1]{Tseng-JOTA2001}. Then, he considered
\begin{equation}\label{eq:obj-extended}
    \min_{\bx} F(\bx) + \sum_{i=1}^b R_i(\bx_i) \ \ \ \ \ \ \ \ \ \  \ \ \ \ \  \bx=[\bx_1;\dots;\bx_b]\in\bbR^{n_1+\cdots + n_b},
\end{equation}
a formulation more general than \cref{eq:obj} in the sense that $R_i(\cdot)$ can be the indicator function of $\cM_i$; be that as it may, a quasi-convexity assumption on the objective of \cref{eq:obj-extended} is still needed in \cite[Theorem 5.1]{Tseng-JOTA2001}.

\myparagraph{Convexity, Convergence Rates} In \textit{\textbf{1993}}, Luo \& Tseng proved local linear convergence of \cref{algo:BEM} in \cite[Proposition 3.4]{Luo-AOR1993}, by assuming, among other assumptions, that $F$ is block-wise strongly convex and $\cM_i$'s are closed intervals of $\bbR$; see also \cite[Theorem 2.1]{Luo-JOTA1992}. More than a decade later, with motivations from many big data applications, a fruitful series of works have been published, which prove, under different settings, global convergence rates of \cref{algo:BEM} and its variants applied to \cref{eq:obj-extended}; see, e.g., \cite{Hsieh-ICML2008,Shalev-ICML2009,Nesterov-SIAM-J-Opt2012,Saha-SIAM-J-Opt2013,Beck-SIAM-J-2013,Richtarik-MP2014,Razaviyayn-NeurIPS2014,Lin-NeurIPS2014,Beck-SIAMOpt2015,Lu-MP2015,Sun-NeurIPS2015,Li-JMLR2017,Hong-MP2017,Diakonikolas-ICML2018,Cai-arXiv2022}; some tricks of the trade for such proofs are timely summarized in \cite[Chapters 11 \& 14]{Beck-OptBook2017}. These exciting contributions, however, come with certain convexity assumptions; for example, by assuming $R_i(\cdot)$ to be convex, one misses out on a whole universe of opportunities for incorporating non-convex regularizers or constraints, which could be valuable in many contexts.

\myparagraph{Non-Convexity, Convergence} For the \textit{fully} non-convex case where no \textit{block-wise} convexity assumption is ever made, related works have been few and far between. In \textit{\textbf{1999}} \cite[Theorem 6.3]{Grippo-OMS1999}, Grippo \& Sciandrone established an asymptotic convergence of \cref{algo:BEM} for the unconstrained case with two blocks; see also \cite[Corollary 2]{Grippo-ORL2000}. A similar convergence result was shown later in \textit{\textbf{2017}} \cite[Theorem 1]{Xu-JSC2017} for a variant of \cref{algo:BEM} applied to formulation \cref{eq:obj-extended}, where Xu \& Yin assumed $F$ in \cref{eq:obj-extended} is block-wise Lipschitz smooth. Finally, it has recently been shown that, in the unconstrained case, \cref{algo:BEM} with (randomized) block gradient descent update escapes strict saddle points almost surely \cite{Lee-MP2019,Chen-SIAM-J-Opt2023}.

\myparagraph{Non-Convexity, Convergence Rates} In  \textit{\textbf{1997}}  \cite[Chapter 3]{Bertsekas-1997}, Bertsekas \& Tsitsiklis showed \cref{algo:BEM} converges linearly under some \textit{contraction mapping} assumption; contraction mappings can be constructed in the strongly convex, Lipschitz smooth, and unconstrained case, but are otherwise hard to verify in the non-convex and constrained regime. In the last decade, several works emerged with global sublinear rate guarantees for \cref{algo:BEM}, but were tailored to specific non-convex applications, e.g., \textit{matrix factorization} (\cref{example:MF}), \textit{optimal transport} \cite[Theorem 1(a) \& Section 5]{Guminov-ICML2021} and \textit{projection robust Wasserstein distance} (\cref{example:PRWD}), \textit{Burer-Monteiro factorization} (\cref{example:SDP}). This motivates the need to unify these proofs under a general, flexible framework.

The work of Xu \& Yin  \cite{Xu-JSC2017} shines for its generality in analyzing a variant of \cref{algo:BEM} applied to \cref{eq:obj-extended}, but their rate guarantee \cite[Theorem 3]{Xu-JSC2017} assumes the so-called \textit{Kurdyka-Lojasiewicz} property \cite[Definition 1]{Xu-JSC2017}. This property is hard to verify and incurs some unknown constant (cf. \cite[Remark 7]{Xu-JSC2017}), which somehow limits the practical light that their theory could have otherwise shed.

To tackle an important application in robotics, Tian et al. considered variants of \cref{algo:BEM} \cite{Tian-TRO2021,Tian-RAL2020}. In \cite[Section IV]{Tian-TRO2021}, Tian et al. used a Riemannian trust-region method to decrease the objective sufficiently over each block. This is a different update rule than \cref{eq:BEM-i}-\cref{eq:BRGD-i}. The other notable difference is that Tian et al. considered selecting blocks in a greedy or stochastic fashion, rather than in the cyclic order of \cref{algo:BEM}. Finally, while Tian et al. instinctively felt that their method and analysis would apply to a wider class of smooth manifold optimization problems beyond robotics, we enrich this instinct with ample examples.

The recent paper of Gutman \& Ho-Nguyen \cite{Gutman-MOR2023} is another effort to extend block coordinate methods to the non-convex case; their Theorem 3 guarantees a sublinear convergence rate for an interesting block variant of \textit{Riemannian} gradient descent. The framework of \cite{Gutman-MOR2023} is pretty general in that $\cM$ is a manifold not necessarily embedded in the Euclidean space and does not necessarily have a product structure as in \cref{eq:obj}, but in pursuing this generality of the problem setup it has sacrificed \textit{flexibility}, and, somehow unexpectedly, \textit{applicability}:
\begin{itemize}[wide,parsep=0pt,topsep=0pt]
    \item (\textit{Flexibility}) \textit{Parallel transport} is used in the analysis of \cite{Gutman-MOR2023}, in combination with \textit{exponential maps}. The use of these advanced Riemannian geometry tools makes it non-trivial to incorporate insights from the vast literature (reviewed above) into the framework of \cite{Gutman-MOR2023} and to extend their algorithm to other commonly used paradigms, e.g., block exact minimization \cref{eq:BEM-i}, block majorization minimization \cref{eq:BMM-i}, block Riemannian gradient descent with general retractions \cref{eq:BRGD-i}, or their combinations. 
    \item (\textit{Applicability}) The general analysis of \cite{Gutman-MOR2023} has assumptions that actually hinder its applicability. Specifically, in \cite{Gutman-MOR2023} $\cM$ is assumed to be \textit{connected}, which fails to be true, e.g., if $\cM$ is the \textit{orthogonal group} known to have two connected components. Also, $F(\bx)$ is assumed to have Lipschitz continuous \textit{Riemannian} gradients in $\bx$ up to parallel transport \cite[Eq. 3]{Gutman-MOR2023}. The assumption is hard to verify due to the presence of parallel transport, and, as will be shown, the Euclidean counterpart of this assumption (which is used in \cite{Beck-SIAM-J-2013}, \cite[Theorem 11.14]{Beck-OptBook2017}) is violated for a class of applications.
\end{itemize}

\section{Our Contribution}\label{section:contribution}
By making no convexity assumptions whatsoever on the constraint set $\cM$ or objective $F$, we land ourselves on the other side of the watershed, where we provide novel analysis of the convergence (rates) for \cref{algo:BEM} and its variants. In our setup \cref{eq:obj}, $\cM$ is a product set $\cM_1\times\cdots \times \cM_b$ embedded in the Euclidean space, where $\cM_i$'s are either potentially non-convex (\cref{section:convergence}) or specifically smooth manifolds (\cref{section:convergence-rate}). In the latter case, our setting is slightly less general than the framework of \cite{Gutman-MOR2023} (reviewed above), but our aim is to trade generality for enhancing other aspects significantly:
\begin{itemize}[wide,parsep=0pt,topsep=0pt]
    \item (\textit{Flexibility}) Our setting and analysis try to balance Riemannian and Euclidean components so that prior works and insights developed over the decades can be easily integrated into the proposed setting. Indeed, our analysis of \cref{algo:BEM} covers a meaningful combination of exact minimization, majorization minimization, and block Riemannian gradient descent with general retractions.
    \item (\textit{Applicability}) Different from \cite{Gutman-MOR2023}, we never assume $\cM$ is connected. Moreover, in \cref{theorem:BEM-RGD}  $F(\bx)$ is not assumed to have Lipschitz continuous (Riemannian) gradients in $\bx$ (up to parallel transport). Using weaker assumptions introduces some technical difficulties that we have overcome, and our theory enjoys wider applicability, as shown by more than twenty examples throughout the paper.
\end{itemize}

\section{Technical Background}\label{section:preliminary}
\myparagraph{Riemannian Geometry} 
Let $\cM$ be a smooth manifold of $\bbR^{n_1+\cdots+n_b}$. Each $\bx\in \cM$ is associated with a linear subspace of $\bbR^{n_1+\cdots+n_b}$, called the \textit{tangent space}, denoted by $T_{\bx}\cM$. A smooth function $F : \cM \to \bbR$ is associated with the gradient $\nabla F(\bx)$ of $F$, as well as the projection of $\nabla F(\bx)$ onto the tangent space $T_{\bx}\cM$; this projection is called the \textit{Riemannian gradient} of $F$ at $\bx$, denoted by $\rgrad F(\bx)$. $\cM$ is also associated with a map $\retr_{\bx}:T_{\bx} \cM\to \cM$, called \textit{retraction}, which satisfies certain smoothness properties (see, e.g., \cite[Deﬁnition 4.1.1]{Absil-2009} for an exact definition). Here are two examples: If $\cM$ is the Euclidean space, then $\bv\mapsto \bx + \bv $ is a retraction; if $\cM$ is the unit sphere, then $\bv\mapsto (\bx + \bv) / \|\bx + \bv \|$ is a retraction (defined on $T_{\bx}\cM=\{\bv: \bv^\top \bx=0 \}$). The reader interested in Riemannian geometry and optimization is referred to books \cite{Absil-2009,Udriste-book1994,Lee-book2012,Lee-book2018,Boumal-2023} and papers \cite{Tron-TAC2012,Afsari-SIAM-J-CO2013,Bonnabel-TAC2013,Liu-ICML2016,Zhang-COLT2016,Bento-JOTA2016,Bento-JOTA2017,Gao-SIAM-J-Opt2018,Boumal-IMA-J-NA2019,Criscitiello-NeurIPS2019,Sun-NeurIPS2019,Chen-SIAM-J-Opt2020,Zhang-MP2020,Huang-MP2022,Weber-MP2022,Criscitiello-FoCS2022,Kim-ICML2022,Li-arXiv2022ADMM,Sra-arXiv2022v2,Xiao-arXiv2022v2,Si-arXiv2023v2}.


\myparagraph{Tangent Cone \& Stationary Points} Recall the definition of \textit{tangent cone} (cf. \cite[Definition 5.2]{Andreasson-2020}):
\begin{definition}[Tangent Cone]\label{def:tangent-cone}
    Let $S\subset\bbR^n$ be a non-empty closed set with $\ba\in S$. A vector $\bv$ is called a \textit{tangent direction} at $\ba$ for $S$ if for some sequence $\{\ba_j\}_j\subset S$ and $\{\lambda_j\}_j\subset (0,\infty)$ we have
    \begin{equation}
        \lim_{j\to\infty} \ba_j = \ba,\ \ \ \ \lim_{j\to\infty} \lambda_j=0 ,\ \ \ \ \lim_{j\to\infty} \frac{\ba_j - \ba}{\lambda_j} = \bv.
    \end{equation}
    The \textit{tangent cone} $T_{S}(\ba)$ of $S$ at $\ba$ is the set of all tangent directions at $\ba$ for $S$.
\end{definition}
Since $F$ is differentiable, every local minimizer $\bx\in \cM$ of \cref{eq:obj} satisfies (cf. \cite[Theorem 6.12]{Rockafellar-2009})
\begin{equation}\label{eq:optimality-condition}
    \langle \nabla F(\bx), \bv \rangle \geq 0,\ \  \forall \bv\in T_{\cM}(\bx).
\end{equation}
\begin{definition}\label{definition:stationary}
    If $\bx\in\cM$ and $\bx$ fulfills \cref{eq:optimality-condition}, then we call $\bx$ a stationary point of \cref{eq:obj}.
\end{definition}
Note that, if $\cM$ is a smooth manifold, then the tangent cone $T_{\cM}(\bx)$ coincides with the tangent space $T_{\bx}\cM$ and the optimality condition \cref{eq:optimality-condition} becomes $\rgrad F(\bx)=0$ (i.e., Riemannian gradient $=$ $0$).

\myparagraph{Notations} The objective $F$ is a function defined on $\bbR^{n_1+\cdots+n_b}$, while we use the same symbol, $F$, to denote its restriction on $\cM$. The symbols $\| \cdot \|$ and $\langle \cdot, \cdot \rangle$ denote Euclidean norm and Euclidean inner product, respectively. Denote by $\nabla_i F(\bx)$ the partial derivative with respect to the $i$-th block of variables $\bx_i$, and by $\rgrad_i F(\bx)$ the projection of $\nabla_i F(\bx)$ onto $T_{\bx_i} \cM_i$. We use $[\bx_1;\dots;\bx_b]$ to denote concatenation into a column vector denoted by $\bx_{1:b}$, also written as $\bx$ for brevity. For $i\in \{1,\dots,b\}$ and $\bxi\in \cM_i$, we will frequently write $[\bx_{1:i-1}; \bxi; \bx_{i+1:b}]$, which means in particular $[\bxi;\bx_{2:b}]$ for the case $i=1$ or $[\bx_{1:b-1};\bxi]$ if $i=b$; that is, if $i>j$, the notation $\bx_{i:j}$ denotes an empty vector.

\section{Convergence}\label{section:convergence}
We make the following basic assumption throughout the section:
\begin{assumption}\label{assumption:basic}
    Every $\cM_i\subset \bbR^{n_i}$ is closed and non-empty. The objective $F:\bbR^{n_1+\cdots+n_b}\to \bbR$ of \cref{eq:obj} is differentiable and its level set $\{ \bx: F(\bx) \leq \gamma \}$ is bounded for every $\gamma\in\bbR$. 
\end{assumption}
\cref{assumption:basic} ensures \cref{eq:obj} is well-posed in the sense that the set of global minimizers of \cref{eq:obj} is non-empty, forming a compact subset of $\cM$ (cf. \cite[Example 1.11]{Rockafellar-2009}).

In this section we derive asymptotic convergence guarantees for \cref{algo:BEM}. Note that at this stage we do not assume $\cM_i$ is a smooth manifold, meaning that we will suppress the Riemannian flavor \cref{eq:BRGD-i} of  \cref{algo:BEM} for the moment and focus on exact minimization  \cref{eq:BEM-i} and majorization minimization \cref{eq:BMM-i}. Specifically, in \cref{subsection:BEM} we consider \cref{algo:BEM} with \cref{eq:BEM-i} executed all the time, while in \cref{subsection:BMM} we analyze step \cref{eq:BMM-i} solely. This is without loss of generality as combining everything together yields guarantees for \cref{algo:BEM} with options \cref{eq:BEM-i} and \cref{eq:BMM-i} blended. Finally, we emphasize that our convergence theory in this section arises as extensions of the classic results, namely \cite[Proposition 2.7.1]{Bertsekas-1999} and \cite[Theorem 2]{Razaviyayn-SIAM-J-Opt2013}. That said, both the results of \cite{Bertsekas-1999} and \cite{Razaviyayn-SIAM-J-Opt2013} assume $\cM_i$'s are convex, while our theorems do not have this convexity assumption at all. 


\subsection{Block Exact Minimization}\label{subsection:BEM}
Here we consider \cref{algo:BEM} with rule \cref{eq:BEM-i}. We need the following assumption:
\begin{assumption}[Unique Block Minimizer]\label{assumption:unique-block-min}
    For every $i=1,\dots,b$ and every $\bx_j\in \cM_j$ ($\forall j\neq i$), the following problem has a unique global minimizer:
    \begin{equation}
        \min_{\bxi\in\cM_i} F(\bx_{1:i-1},\bxi, \bx_{i+1:b})
    \end{equation}
\end{assumption}
\begin{remark}[\cref{assumption:unique-block-min} in GPCA]\label{remark:unique-block-min-GPCA}
    \cref{eq:GPCA-LS} immediately violates \cref{assumption:unique-block-min}: If $\bA_i^{t+1}$  minimizes \cref{eq:BEM-i-GPCA}, then right-multiplying $\bA_i^{t+1}$ by any orthonormal matrix yields a global minimizer as well. This seems to be an inherent issue as \cref{eq:BEM-i-GPCA} admits a unique global minimizer only up to isometry. This can be fixed fairly easily: revise \cref{assumption:unique-block-min} such that the uniqueness is up to the symmetry of the underlying manifolds. With this in mind, we proceed with \cref{assumption:unique-block-min} for simplicity.
\end{remark}
For illustration, assume $\cM_i=\bbR^{n_i}$. If $F$ is block-wise strictly convex, then \cref{assumption:unique-block-min} holds, while the converse is not necessarily true, therefore  \cref{assumption:unique-block-min} is a weaker assumption than block-wise strict convexity. However, in the constrained case, \cref{assumption:unique-block-min} does not necessarily hold even if $F$ is block-wise strictly convex, due partly to the potential symmetry (cf. \cref{remark:unique-block-min-GPCA}).

A basic convergence guarantee for \cref{algo:BEM} now ensues:
\begin{theorem}\label{theorem:BEM}
    Under \cref{assumption:basic} and \cref{assumption:unique-block-min}, the sequence $\{\bx^t\}_t$ of \cref{algo:BEM} that performs updates via \cref{eq:BEM-i} is bounded and has limit points. Moreover, every limit point is a stationary point of \cref{eq:obj}.
\end{theorem}

\subsection{Block Majorization Minimization}\label{subsection:BMM}
Here, we consider \cref{algo:BEM} with the majorization minimization option \cref{eq:BMM-i} activated. Note that option \cref{eq:BMM-i} requires some upper bound function (\textit{majorizer}) $G_i$ of $F$ for block $i$. Therefore, to proceed, we need first to make the notion of (\textit{block}) \textit{majorizer} precise:
\begin{definition}[Block-$i$ Majorizer]\label{definition:BM}
    Function $G_i:\cM_i\times \cM\subset \bbR^{n_i}\times \bbR^{n_1+\cdots+n_b}\to \bbR$ is called a block-$i$ majorizer of $F$ if the following hold for every $\bx_i\in\cM_i$ and every $\bz\in\cM$:
    \begin{itemize}
        \item $F(\bz_{1:i-1}, \bx_i, \bz_{i+1:b}) \leq  G_i(\bx_i, \bz)$, with equality attained when $\bx_i=\bz_i$.
        \item $G_i$ is differentiable and the two gradients $\nabla_{\bz_i} F(\bz)$ and $\nabla_{\bx_i} G_i(\bx_i, \bz)|_{\bx_i=\bz_i}$ are equal.
    \end{itemize}
\end{definition}
\begin{example}[\cref{definition:BM} in GPCA-Huber]\label{example:majorizer-GPCA}
    Consider the functions
    \begin{align*}
        g(r,s)&:= h(s) + \frac{r^2 - s^2}{2 \cdot \max\{|s|, \epsilon \} } \\ 
        G_1(\bC_1,\bA_1,\bA_2)&:= \sum_{j=1}^m g\big( \| \bp_j^\top \bC_1 \| ; \| \bp_j^\top \bA_1 \| \big) \cdot h\big( \| \bp_j^\top \bA_2 \| \big).
    \end{align*}
    One verifies that, for every $r,s\in\bbR$, we have  $h(r)\leq g(r,s)$ and the partial derivative of $\frac{\partial g(r,s)}{\partial r}$ at $r=s$ is precisely $h'(s)$. Therefore, $G_1$ is a block-$1$ majorizer of the objective of \cref{eq:GPCA-Huber}; a block-$2$ majorizer can be constructed similarly. While it is not obvious how to solve \cref{eq:GPCA-Huber}, minimizing majorizer $G_1$ in variable $\bC_1$ is an eigenvalue problem similar to \cref{eq:BEM-i-GPCA}, easy to solve.
\end{example}
The construction of (block-wise) majorizers is problem-specific, but general principles and practical recipes for doing so have been explored; see, e.g.,  \cite[Propositions 1 and 2]{Razaviyayn-SIAM-J-Opt2013}, \cite{Sun-TSP2016,Hunter-TAS2004,Lange-MM2016,Sra-arXiv2022v2}. To provide a general analysis, we will assume that $F$ admits some block-wise majorizers (instead of constructing them). For majorization minimization \cref{eq:BMM-i} to make sense, it is implicitly assumed that the block-wise majorizers are much easier to be minimized than $F$ over block $i$ (cf. \cref{example:majorizer-GPCA}).

	


The following convergence result is of a similar flavor to \cref{theorem:BEM}:
\begin{theorem}\label{theorem:BMM}
    Assume \cref{assumption:basic} holds and $F$ has a block-$i$ majorizer $G_i$ for every $i=1,\dots,b$. Moreover, for every $\bz\in \cM$, assume the following problem has a unique global minimizer:
    \begin{equation}
        \min_{\bxi\in\cM_i} G_i( \bxi, \bz )
    \end{equation}
    If \cref{algo:BEM} performs updates via \cref{eq:BMM-i}, then every limit point of its iterates $\{\bx^t\}_t$ is a stationary point of \cref{eq:obj}.
\end{theorem}
\begin{example}[Matrix Factorization]\label{example:MF}
With hyper-parameters $\lambda >0,p\in(0,1]$, and $\eta>0$, consider the following formulation for low-rank matrix factorization (cf. \cite[Eq. 22]{Giampouras-TPS2018}):
    \begin{equation}
        \min_{\bU,\bV} F(\bU, \bV) \ \ \ \ \ \ \  F(\bU, \bV):= \frac{1}{2} \cdot \|\bY-\bU\bV^\top  \|_{\textnormal{F}}^2 + \lambda \sum_{i} \big( \|\bu_i \|^2 + \| \bv_i \|^2 +\eta^2 \big)^{p/2}
    \end{equation}
The above summation is over the columns $\bu_i$ (resp. $\bv_i$) of $\bU$ (resp. $\bV$). In \cite{Giampouras-TPS2018}, \cref{algo:BEM} was employed with customized block-wise majorizers for minimization in variables $\bU$ and $\bV$ in an alternating style. Since their majorizers admit unique minimizers (cf. \cite[Remark 2]{Giampouras-TPS2018}), \cref{theorem:BMM} can be invoked to obtain a convergence guarantee for their algorithm, recovering \cite[Proposition 3(a)]{Giampouras-TPS2018}. While $\bU$ and $\bV$ are unconstrained in \cite{Giampouras-TPS2018}, \cref{theorem:BMM} can further be applied to the case where $\bU$ or $\bV$ is constrained to lie in the Stiefel manifold \cite{Dai-TIT2012,Meyer-ICML2011,Boumal-NeurIPS2011,Vandereycken-SIAM-J-Opt2013,Mishra-CS2014} and to \textit{tensor decomposition} \cite{Rontogiannis-JSTSP2021,Giampouras-TSP2022}.
\end{example}
\cref{algo:BEM} executing \cref{eq:BMM-i} minimizes a different objective $G_i$, and yet it converges to a stationary point of \cref{eq:obj} whose objective is $F$. This is not surprising given the connections between $G_i$ and $F$ (\cref{definition:BM}). Indeed, the gradients $G_i$ and $F$ coincide at a certain point, by which one can ``transfer'' the stationarity condition of  \cref{eq:BMM-i} to that of \cref{eq:BEM-i}, and this is how \cref{theorem:BMM} can be proved. Further analysis and additional results regarding \cref{algo:BEM} with option \cref{eq:BMM-i} can be found in the appendix.

\section{Convergence Rates}\label{section:convergence-rate}
Our theorems of \cref{section:convergence} hold for general (non-convex) constraint sets but are concerned only with asymptotic convergence, not convergence rates. In this section, we assume each $\cM_i\subset \bbR^{n_i}$ is a smooth manifold. By endowing such a smoothness structure on $\cM_i$, we can leverage Riemannian optimization techniques and derive more informative convergence rates under verifiable conditions. 

Similarly to \cref{section:convergence} and to begin with, we consider \cref{algo:BEM} that updates $\bx_i^{t+1}$ the Riemannian gradient descent option \cref{eq:BRGD-i} in \cref{subsection:BRGD}. Moreover, in \cref{subsection:BEM+BRGD}, we will allow for the simultaneous use of block exact minimization \cref{eq:BEM-i} and Riemannian gradient descent \cref{eq:BRGD-i} as this fits the practical needs (\cref{example:PRWD}). 

\subsection{Block Riemannian Gradient Descent }\label{subsection:BRGD}
We begin by generalizing the usual Lipschitz smoothness into a block-wise definition:
\begin{definition}[Block-$i$ Lipschitz Smoothness]\label{definition:block-Lipschitz}
Let $i\in\{1,\dots,b\}$. We say $F$ is block-$i$ Lipschitz smooth with constant $L_i>0$ if for every $[\bx_{1:i-1}; \bxi; \bx_{i+1:b}]\in\cM$ and every $\bzeta\in \cM_i$ we have
\begin{equation}\label{eq:block-Lipschitz-continuous-grad}
        \| \nabla_i F ( \bx_{1:i-1}, \bxi ,\bx_{i+1:b}  ) - \nabla_i F ( \bx_{1:i-1}, \bzeta ,\bx_{i+1:b}  ) \| \leq  L_i\cdot \| \bxi - \bzeta  \|.
    \end{equation}
\end{definition}
\begin{example}[\cref{definition:block-Lipschitz} in GPCA]\label{example:Lipschitz-GPCA}
    The objective $F(\bA_1,\bA_2):=\sum_{j=1}^m \| \bp_j^\top \bA_1 \|^2 \cdot \| \bp_j^\top \bA_2 \|^2$ of \cref{eq:GPCA-LS} is block-$1$ (and block-$2$) Lipschitz smooth. To verify this, observe 
    \begin{equation}
        \nabla_1 F(\bxi,\bA_2) - \nabla_1 F(\bzeta,\bA_2) =2 ( \bxi - \bzeta ) \sum_{j=1}^m \| \bp_j^\top \bA_2 \|^2\cdot \bp_j \bp_j^\top.
    \end{equation}
    Since $\bA_2$ and data points $\bp_j$'s are bounded, we see that \cref{eq:block-Lipschitz-continuous-grad} holds for a suitable $L_i$. 
\end{example}
	

Crucially, \cref{example:Lipschitz-GPCA} implies that, for $F$ to be \textit{block-$i$} Lipschitz smooth, in general we need to assume that \textit{all} $\cM_j$'s ($j\neq i$) are bounded, which means they are contained in some compact subset of the Euclidean space. In particular, if $\cM_i$ is a compact submanifold of $\bbR^{n_i}$, the (Euclidean) block-wise Lipschitz smoothness of $F$ has the following ramifications for (Riemannian) retractions:
\begin{lemma}\label{lemma:block-Lipschitz}
    Let $i\in\{1,\dots,b \}$ be fixed. Assume $\cM_i$ is a compact submanifold of $\bbR^{n_i}$. Then we have
    \begin{equation}\label{eq:alpha_i}
        \|\retr_{\bxi}(\bs_i) - \bxi \| \leq \alpha_i \cdot \| \bs_i \|, \ \ \ \forall \bxi\in \cM_i,\ \forall \bs_i\in T_{\bxi} \cM_i
    \end{equation}
    for some constant $\alpha_i>0$. Moreover, if $F$ is continuously differentiable and block-$i$ Lipschitz smooth with constant $L_i>0$, and all $\cM_j$'s ($j\neq i$) are compact, then there is a constant $\rL_i>0$ such that for every $[\bx_{1:i-1}; \bxi; \bx_{i+1:b}]\in\cM$ and every $\bs_i\in T_{\bxi} \cM_i$ we have
    \begin{equation}\label{eq:rL_i}
        \begin{split}
            F(\bx_{1:i-1},\retr_{\bxi}(\bs_i),\bx_{i+1:b}) &\leq F(\bx_{1:i-1},\bxi,\bx_{i+1:b}) \\ 
        &\ \ \ \  \ \ \ \ + \langle \rgrad_i F(\bx_{1:i-1},\bxi,\bx_{i+1:b}), \bs_i \rangle + \frac{\rL_i}{2} \| \bs_i \|^2.
        \end{split} 
    \end{equation}
\end{lemma} 
\cref{lemma:block-Lipschitz} is an extension of \cite[Appendix B]{Boumal-IMA-J-NA2019} into multiple blocks. Specifically, \cref{eq:alpha_i} follows directly from \cite[Appendix B]{Boumal-IMA-J-NA2019}, while proving \eqref{eq:rL_i} requires a modification of their proofs. It is understood that the constants $\alpha_i$ and $\rL_i$ of \cref{lemma:block-Lipschitz} depend on $L_i$, $\cM_i$'s, and the retraction, and they are independent of $\bs_i$ and $\bx_i$'s. To proceed, we define some constants based on \cref{lemma:block-Lipschitz}. Define $L^\textnormal{max}_b:= \max_{i=1,\dots,b} L_i$, and similarly define $\alpha^\textnormal{max}_b$, $\rL^\textnormal{max}_b$, and $\rL^\textnormal{min}_b$. Finally, defining the quantity
 \begin{equation}\label{eq:C}
         C_b:=  \sqrt{2  \rL^\textnormal{max}_b } + \alpha^\textnormal{max}_b L^\textnormal{max}_b \cdot \sqrt{ \frac{2b}{\rL^\textnormal{min}_b }  },
 \end{equation}
we can then state that the iterates of \cref{algo:BEM} with option \cref{eq:BRGD-i} gravitate towards stationarity:
\begin{theorem}\label{theorem:RGD}
     For each $i=1,\dots,b$ assume $\cM_i$ is a compact smooth submanifold of $\bbR^{n_i}$ and $F$ is block-$i$ Lipschitz smooth with constant $L_i$. Let $C_b$ be defined in \cref{eq:C} and $\rL_i$ in \cref{lemma:block-Lipschitz}. If \cref{algo:BEM} updates its block via \cref{eq:BRGD-i} with stepsize $\lambda^t_i=1/\rL_i$, then its iterates $\{\bx^t\}_{t}$ satisfy
 \begin{equation}\label{eq:BRGD-sublinear}
     \min_{t=0,\dots,T} \big\| \rgrad F(\bx^{t}) \big\| \leq \sqrt{b}\cdot C_b\cdot \sqrt{ \frac{F(\bx^0) - F(\bx^{T+1}) }{T+1} }.
 \end{equation}
\end{theorem}
\cref{theorem:RGD} assumes a constant stepsize $1/\rL_i$ for block $i$ for simplicity. When $\rL_i$ is unknown, one could use line search to get similar convergence guarantees. \cref{theorem:RGD} further assumes that $F$ is block-$i$ Lipschitz smooth with constant $L_i>0$ for every $i=1,\dots,b$. This implies $F$ is globally Lipschitz smooth in the sense that $\|\nabla F(\bx) - \nabla F(\by) \|\leq L\cdot \| \bx - \by \| $ for some constant $L$ with $0<L\leq \sum_{i=1}^b L_i$ \cite[Lemma 2]{Nesterov-SIAM-J-Opt2012}. However, in some applications, $F$ is not Lipschitz smooth block-wise, let alone globally:
\begin{example}[Geman-McClure Loss, cont.]\label{example:GM2}
    Recall \cref{example:GM} where the Geman-McClure loss \cref{eq:GM} corresponds to the regularizer $R_\rho(w_i)=(\sqrt{w_i}-1)^2$ in the outlier-robust estimation problem \cref{eq:weight+x}. We have
    \begin{equation*}
        |R_\rho'(w_i) - R_\rho'(1)| =  \Big| \frac{1}{\sqrt{w_i}} - 1 \Big|, \quad \forall w_i\in (0,1],
    \end{equation*}
    so $|R_\rho'(w_i) - R_\rho'(1)|$ tends to $\infty$ as $w\to 0$, whereas $L\cdot |w_i-1|$ is bounded above by $L$ for any finite constant $L>0$ and any $w\in(0,1]$. Hence $R_\rho(w_i)$ is not Lipschitz smooth in $w_i$, neither is the objective of \cref{eq:weight+x}.
\end{example}
In \cref{subsection:BEM+BRGD}, we will derive theorems where the block-wise Lipschitz smoothness is not needed for a given block, and in this way we dispense with the global Lipschitz smoothness assumption, making our results applicable to a broader class of problems including outlier-robust estimation with the Geman-McClure loss.

\subsection{Blending Block Exact Minimization and Riemannian Gradient Descent }\label{subsection:BEM+BRGD}
We are sometimes forced to blend block exact minimization and Riemannian gradient descent: 
\begin{example}[Projection Robust Wasserstein Distance, PRWD]\label{example:PRWD} As discussed in \cite{Huang-ICML2021}, the PRWD problem involves transporting one (discrete) distribution into another, and it has a lot of machine learning applications. It was shown that the PRWD problem can be formulated as (cf. \cite[Eq. 14]{Huang-ICML2021})
\begin{equation}\label{eq:PRWD}
    \min_{(U,v,u)\in \cM_1\times \bbR^{n_2}\times \bbR^{n_3}  } g(U,v,u),
\end{equation}
where $\cM_1$ is the Stiefel manifold. To solve \cref{eq:PRWD}, \cite{Huang-ICML2021} proposes a block coordinate descent variant, which performs exact minimization over variables $u$ and $v$ and Riemannian gradient descent over $U$. The rationale behind their proposal is that minimizing $g(U,v,u)$ in variable $U\in \cM_1$ does not admit a closed-form solution, and one then has to appeal to the manifold structure of $\cM_1$ for a descent.
\end{example}
\cref{algo:BEM} is precisely motivated by the scenario of \cref{example:PRWD} that calls for blending different block update rules. To fully simulate this scenario, we specialize \cref{algo:BEM} into \cref{algo:BEM-RGD}. \cref{algo:BEM-RGD} requires an extra input, $\cR\subset\{ 1,\dots,b-1\}$, a set of indices chosen based on a user's expertise. Blocks corresponding to indices of $\cR$ are updated by Riemannian gradient descent \cref{eq:BRGD-i}, or otherwise by exact minimization \cref{eq:BEM-i}. Since $b\notin \cR$, exact minimization occurs at least once (e.g., for updating block $b$). Finally, we note that in Line 2 of \cref{algo:BEM-RGD}, the slightly curious initialization $\bx_b^0$ is not necessary for convergence but is employed to simplify the presentation. We are now ready to state the following result:

\begin{algorithm}[t]
	\SetAlgoLined
	\DontPrintSemicolon
        Input: $\cR\subset \{ 1,\dots,b-1 \}$: indices of blocks to be updated by Riemannian gradient descent
 
	Initialize $\bx^0_{1:b-1}\in \cM_1\times \cdots\times \cM_{b-1}$; given $\bx^0_{1:b-1}$, compute $\bx^0_b$ via Exact Minimization \cref{eq:BEM-i}
	
	For $t\gets 0,1,\dots, T$: 

        \ \ \ \ \ \ For $i\gets1,\dots,b$: 
        
        \ \ \ \ \ \ \ \ \ \ \ \  If $i\in \cR$:\ \ \ \ \ \ \ \ \ \ \ \ \ \ \ \ \ \  Compute $\bx_i^{t+1}$ via Riemannian Gradient Descent \cref{eq:BRGD-i}

        \ \ \ \ \ \ \ \ \ \ \ \ Else:\ \ \ \ \ \ \ \ \ \ \ \ \ \ \ \ \ \ \ \ \ \ \ \ \  Compute $\bx_i^{t+1}$ via Exact Minimization \cref{eq:BEM-i}
        
	\caption{Blending Block Exact Minimization and Riemannian Gradient Descent } 
 \label{algo:BEM-RGD}
\end{algorithm}

\begin{theorem}\label{theorem:BEM-RGD}
    For every $i=1,\dots,b$, assume $\cM_i$ is a compact smooth submanifold of $\bbR^{n_i}$. Assume $F$ is continuously differentiable and block-$i$ Lipschitz smooth for every $i=1,\dots,b-1$ (not necessarily for $i=b$). The iterates $\{\bx^t\}_{t}$ of \cref{algo:BEM-RGD} with stepsize $\lambda^t_i=1/\rL_i$ if $i\in \cR$ satisfy
 \begin{equation}\label{eq:BEM-RGD-rate}
     \min_{t=0,\dots,T} \big\| \rgrad F(\bx^{t}) \big\| \leq \sqrt{b-1}\cdot C_{b-1}\cdot \sqrt{ \frac{F(\bx^0) - F(\bx^{T+1}) }{T+1} } \ \ \ \ \textnormal{  with $C_{b-1}$ defined in  \cref{eq:C}}.
 \end{equation}
\end{theorem}
\cref{theorem:BEM-RGD} is based on \cref{theorem:RGD}, and one of the key ideas in proving \cref{theorem:BEM-RGD} is that that exact minimization \cref{eq:BEM-i} always entails a larger (if not equal) decrease of the objective than a descent step \cref{eq:BRGD-i}. \cref{theorem:BEM-RGD} is significant as it covers several existing theoretical results as special cases and yields novel results for multiple other applications. More examples are now in order.
\begin{example}[PRWD, cont.]\label{example:PRWD-cont}
    Recall \cref{example:PRWD}. In the PRWD problem \cref{eq:PRWD}, the objective $g(U,v,u)$ is Lipschitz smooth with respect to $U$ (as proved in \cite[Lemma C.6]{Huang-ICML2021}) and also $v$ (as one could verify from the definition of $g$ in \cite{Huang-ICML2021}). Therefore, \cref{theorem:BEM-RGD} can be applied to the algorithm of \cite{Huang-ICML2021}, yielding a rate guarantee similar to \cite[Theorem 3.3]{Huang-ICML2021}. It should be noted that proving \cite[Theorem 3.3]{Huang-ICML2021} in its exact form requires application-specific details, which we will not explore there.
\end{example}
When $\cR=\varnothing$, \cref{theorem:BEM-RGD} yields a convergence rate guarantee for block exact minimization. Such guarantee can be found in \cite{Erdogdu-MP2022} and \cite{Tian-arXiv2019v4}, but both of the papers \cite{Erdogdu-MP2022,Tian-arXiv2019v4} are specifically for the \textit{Burer-Monteiro factorization} \cref{eq:BM} of some diagonally constrained semidefinite program \cref{eq:SDP}:
\begin{example}[Burer-Monteiro Factorization]\label{example:SDP}
Consider the following pair of problems:

\begin{minipage}{.49\textwidth}
  \begin{equation}
      \begin{split}
          \min_{\bX: \bX\succeq 0 } &\ \   \trace(\bC \bX)  \\ 
          \textnormal{s.t.} & \ \ [\bX]_{ii} = \bI_d, \forall i=1,\dots,n
      \end{split} \tag{SDP} \label{eq:SDP}
  \end{equation}
\end{minipage}%
\begin{minipage}{.49\textwidth}
  \begin{equation}
      \begin{split}
         &\ \  \min_{\bY: \bY=[\bY_1,\dots,\bY_n] } \ \trace(\bC \bY^\top \bY  ) \\ 
          \textnormal{s.t.} & \ \ \bY_i^\top \bY_i = \bI_d, \forall i=1,\dots,n
      \end{split} \tag{BM} \label{eq:BM}
  \end{equation}
\end{minipage}

In \cref{eq:SDP}, $[\bX]_{ii}$ denotes the submatrix of $\bX\in\bbR^{nd\times nd}$ whose columns and rows are indexed by $\{ d(i-1)+1,\dots, di \}$; i.e., $[\bX]_{ii}$ is the $i$-th block diagonal submatrix of $\bX$. Substituting $\bX=\bY^\top \bY$ into \cref{eq:SDP} with $\bY\in \bbR^{r\times nd}$ ($r\leq nd$) furnishes its \textit{Burer-Monteiro factorization} \cref{eq:BM} \cite{Burer-MP2003,Burer-MP2005}. 

\cref{algo:BEM-RGD}, minimizing one block of variables $\bY_i$ at a time, can be applied to \cref{eq:BM}, and so does \cref{theorem:BEM-RGD}. In particular, the case $d=1$ models \textit{Max-Cut} and related problems \cite{Goemans-JACM1995,Bandeira-COLT2016,Mei-COLT2017,Erdogdu-Neur2017}; applying \cref{theorem:BEM-RGD} to this case yields the same sublinear rates as \cite[Theorem 1]{Erdogdu-MP2022} (up to constants). The case $r=d>1$ models the problem of \textit{rotation averaging} \cite{Wang-IMA2013,Eriksson-TPAMI2019,Dellaert-ECCV2020,Parra-CVPR2021,Chenb-CVPR2021,Doherty-ICRA2022}, which is otherwise called \textit{rotation synchronization} \cite{Boumal-arXiv2015,Ling-SIAM-J-Opt2022,Gao-J-IMA2023,Zhu-arXiv2023}; applying \cref{theorem:BEM-RGD} yields the same sublinear rates as \cite[Theorem 1]{Tian-arXiv2019v4} (up to constants). Our theory  furthermore applies to other related problems, including \textit{permutation synchronization} \cite{Huang-CGF2013,Leonardos-ICRA2017,Birdal-CVPR2019,Li-CVPR2022,Ling-ACHA2022,Nguyen-arXiv2023}, \textit{generalized orthogonal Procrustes} \cite{Ten-1977,Pumir-J-IMA2021,Ling-ACHA2023}, \textit{generalized canonical correlation analysis} \cite{Hanafi-CSDA2006}, \cite[Example 2.2]{Lange-arXiv2021}, and \textit{orthogonal trace-sum maximization} \cite{Won-SIAM-J-MAA2021,Won-SIAM-J-Opt2022}.
\end{example}

Since \cref{algo:BEM-RGD} performs exact minimization at least once for block $b$, convergence guarantees \cref{eq:BEM-RGD-rate} can be derived without Lipschitz-type assumptions on block $b$, as \cref{theorem:BEM-RGD} indicates. This addresses an issue in \cref{example:GM2}. Also due to this, we can generate novel guarantees from \cref{theorem:BEM-RGD} for the IRLS algorithm applied to outlier-robust estimation (\cref{example:IRLS}). 

We finish with two examples where the assumptions of \cref{theorem:BEM-RGD} are not exactly fulfilled.\footnote{There appear to be 23 examples rather than what the title indicated. We found twenty-one is a good number for two reasons: Once it is reached, one could have a party for celebrations and alcohol; it echoes a Chinese saying, ``\textit{regardless of three, seven, and twenty-one}'', which means ``\textit{regardless of the consequences}'' or ``\textit{throw caution to the wind}''.  }


\begin{example}[Smoothed $\ell_1$ Loss, cont.]
    In \cref{example:smooth-L1}, the weight $w_i$ is constrained to lie $(0,\infty)$, which is not compact. This can be fixed as follows. Since the optimal weight in \cref{eq:specific-IRLS-weight} is bounded above by $1/\epsilon$, so \cref{eq:specific-IRLS} is actually equivalent to
    \begin{equation}\label{eq:specific-IRLS2}
        \min_{\bx \in \cC, w_i\in (0, 1/\epsilon] }  \sum_{i=1}^m \Big( w_i\cdot r_i^2(\bx) + \epsilon^2 \cdot w_i + \frac{1}{w_i} \Big)
    \end{equation}
    The constraint set $(0, 1/\epsilon]$ is bounded, contained in the compact set $[0,1/\epsilon]$. Furthermore, one can weaken the compactness assumption originating from \cref{lemma:block-Lipschitz} by a boundedness assumption, and revise \cref{lemma:block-Lipschitz}. With this revision, a convergence rate similar to \cref{theorem:BEM-RGD} and applicable to \cref{eq:specific-IRLS2} follows.
\end{example}
\begin{example}[Point Cloud Registration, cont.]
    Note that the constraint set $\SE(3)$ in \cref{example:PCR} is not compact, violating an assumption in \cref{theorem:BEM-RGD}. However, it is known that \cref{lemma:block-Lipschitz} holds with $\alpha_i=1$ and $\rL_i=L_i$ if $\cM_i=\bbR^{n_i}$ and $F$ is  block-$i$ Lipschitz smooth. With this fact, following the same proof idea one can verify that a convergence guarantee similar to \cref{theorem:BEM-RGD} still holds if $\cC$ is the product of a compact submanifold and the Euclidean space, e.g., $\SE(3)$. In this way, we can obtain convergence rate guarantees of IRLS applied to \cref{eq:weight+x} for the point cloud registration problem (say, with the Geman-McClure loss or smoothed $\ell_1$ loss). 
\end{example}

\section{Conclusion}\label{section:concolusion}
We laid out a framework for block coordinate descent on embedded submanifolds of the Euclidean space. Within this framework, we provided novel convergence (rate) analysis that can be specialized for multiple applications, either recovering existing results or yielding new convergence guarantees. What enabled our analysis is the formulation \cref{eq:obj} that is situated at a proper level of abstraction and maintains an under-explored balance between generality and flexibility/applicability. 

We have barely scratched the surface, and limitations exist in our current analysis. First of all, it is unlikely that the constant $C_b$, defined in \cref{eq:C} and appearing in \cref{theorem:RGD}, is optimal. It is more unlikely to be so when specializing \cref{theorem:RGD} to applications (even though the overall rate could match). Then it is also important to note that we only considered selecting blocks to minimize in a cyclic manner. While this is the most common and natural, choosing blocks in a greedy or stochastic style is also possible, and analyzing such situations is left as an extension of the present manuscript.




\bibliography{Liangzu}


\appendix
\section{Appendix}


\section{Proofs of Main Theorems}\label{section:proof-theorems}
\subsection{Proof of \cref{theorem:BEM}}
\begin{definition}\label{def:cwm}
    A point $\hat{\bx}\in\cM$ is called a \textit{coordinate-wise minimizer} of \cref{eq:obj} if
    \begin{equation}\label{eq:cwm}
        F(\hat{\bx}) \leq F( \hat{\bx}_{1:i-1}, \bxi, \hat{\bx}_{i+1} ), \ \ \ \forall \bxi\in \cM_i,\ \forall i=1,\dots,b.
    \end{equation}
\end{definition}
Below we prove a slightly stronger result than \cref{theorem:BEM}.
\begin{theorem}
    Under \cref{assumption:basic} and \cref{assumption:unique-block-min}, the sequence $\{\bx^t\}_t$ produced by \cref{algo:BEM} is bounded and has limit points. Moreover, every limit point $\hat{\bx}$ of $\{\bx^t\}_t$ is a coordinate-wise minimizer of \cref{eq:obj}. Therefore, by the differentiability of $F$ and \cref{prop:cwm->cws->s}, $\hat{\bx}$ is a stationary point of \cref{eq:obj}.
\end{theorem}

\begin{proof}
By \cref{eq:BEM-i}, we have
\begin{equation}\label{eq:non-increasing}
    F(\bx^{t+1}) = F(\bx^{t+1}_{1:b})\leq  F(\bx^{t+1}_{1:b-1}, \bx^{t}_b ) \leq \cdots \leq F(\bx^{t+1}_1, \bx^{t}_{2:b}) \leq F(\bx^{t}_{1:b})= F(\bx^t).
\end{equation}
Since $F$ is continuous, its level set is closed \cite[Theorems 2.6 and 2.8]{Beck-OptBook2017}. Since the level set of $F$ is bounded by \cref{assumption:basic}, it is compact, and \cref{eq:non-increasing} implies $F(\bx^t)$ converges to some real number say $\hat{F}$ \cite[Theorem 2.12]{Beck-OptBook2017}. In fact, \cref{eq:non-increasing} implies $F(\bx^{t+1}_{1:i}, \bx^{t}_{i+1:b})$ also converges to $\hat{F}$, for every $i$.

The boundedness of the level set and \cref{eq:non-increasing} implies $\{\bx^t\}_t$ is a bounded sequence. It now follows that $\{\bx^t\}_t$ has a convergent subsequence; let $\{ \bx^{t_j} \}_j$ be such subsequence convergent to some limit point, say $\hat{\bx}=\hat{\bx}_{1:b}$. Passing to subsequences if necessary, we assume that for every $i=1,\dots,b$ the sequence $\{ [\bx^{t_j+1}_{1:i-1}; \bx^{t_j}_{i:b}] \}_j$ converges to some vector $[\hat{\by}_{1:i-1};\hat{\bx}_{i:b}]$. Since $\cM$ is closed, we know $\hat{\bx}\in \cM$ and $[\hat{\by}_{1:i-1};\hat{\bx}_{i:b}]\in \cM$. In particular, we have $F(\hat{\by}_{1:i-1}, \hat{\bx}_{i:b})=\hat{F}$ for every $i$.

Next, we prove $\hat{\by}_i= \hat{\bx}_i$ for every $i$. For the case $i=1$, it follows from \cref{eq:BEM-i} that
\begin{equation*}
    F(\bx^{t_j+1}_1, \bx^{t_j}_{2:b}) \leq  F( \bxi,\bx^{t_j}_{2:b}), \ \ \ \forall \bxi\in \cM_1.
\end{equation*}
Since $F$ is continuous, with $t_j\to\infty$ the above yields
\begin{equation*}
  F(\hat{\bx}) =  \hat{F} =  F(\hat{\by}_1,\hat{\bx}_{2:b}) \leq F(\bxi, \hat{\bx}_{2:b}), \ \ \  \forall \bxi\in \cM_1.
\end{equation*}
Since $F(\bxi, \hat{\bx}_{2:b})$ as a function of $\bxi$ has a unique minimizer over $\cM_1$ attaining the value $\hat{F}$ (\cref{assumption:unique-block-min}), it can only be that $\hat{\by}_1=\hat{\bx}_1$. For the case $i=2$, we get for every $\forall \bxi\in \cM_2$ that
\begin{equation*}
    \begin{split}
        F(\hat{\bx})=\hat{F} &= F( \hat{\bx}_1,\hat{\by}_2,\hat{\bx}_{3:b} ) \\
        &=F( \hat{\by}_1,\hat{\by}_2,\hat{\bx}_{3:b} ) \\
        &\leq F(  \hat{\by}_1,\bxi,\hat{\bx}_{3:b} ) \\
        &\leq F( \hat{\bx}_1,\bxi,\hat{\bx}_{3:b} )
    \end{split}
\end{equation*}
which together with \cref{assumption:unique-block-min} implies $\hat{\bx}_2=\hat{\by}_2$. Proceeding similarly, we can prove $\hat{\bx}_i=\hat{\by}_i$  for every $i$ and
\begin{equation*}
    F(\hat{\bx}) \leq F(\hat{\bx}_{1:i-1},  \bxi, \hat{\bx}_{i+1:b}), \ \ \  \forall \bxi\in \cM_i.
\end{equation*}
Thus $\hat{\bx}$ is a coordinate-wise minimizer of \cref{eq:obj}. Since $F$ is differentiable, \cref{prop:cwm->cws->s} implies $\hat{\bx}$ is also a stationary point of \cref{eq:obj}. The proof is therefore complete.
\end{proof}

\subsection{Proof of \cref{theorem:BMM}}\label{subsection:proof:BMM}
\begin{proof}
    With the definitions $\bq_i^t:=[\bx^{t+1}_{1:i-1}; \bx^t_{i:b}]$ of \cref{eq:BMM-i}  and $G_i$ we have for any $t$ that
    \begin{equation}\label{eq:non-increasing-G}
        \begin{split}
            F(\bx^{t+1}) &\leq G_{b}(\bx^{t+1}_b, \bq^t_b) \\
            &\leq G_{b}(\bx^t_b, \bq^t_b) \\
            &= F(\bq^{t}_b) \\ 
            &\leq G_{b-1}( \bx^t_{b-1},\bq^t_{b-1} ) \\
            &\leq \cdots \\ 
            &= F(\bx^t)
        \end{split}
    \end{equation}
    Therefore, $\{F(\bx^t)\}_t$ is a non-increasing sequence. Similarly to the proof of \cref{theorem:BEM}, we can show under \cref{assumption:basic} that, as $t\to \infty$,  $F(\bq^t_i)$ converges to $\hat{F}$ for every $i$. Moreover, as shown similarly in the proof of \cref{theorem:BEM}, $\{\bx^t\}_t$ is bounded, and it has a convergent subsequence $\{ \bx^{t_j} \}_j$ such that  $\{ \bq^{t_j}_i \}_j=\{ [\bx^{t_j+1}_{1:i-1}; \bx^{t_j}_{i:b}] \}_j$ converges to some vector $\hat{\bq}_i:=[\hat{\by}_{1:i-1};\hat{\bx}_{i:b}]$ for every $i$.

    Note that $F$ and $G_i$ are continuous. Replacing $t$ with $t_j\to\infty$ in \cref{eq:non-increasing-G}, we get for every $i$ that
    \begin{equation}\label{eq:tmp-1231243544235}
        \hat{F} = G_i(\hat{\by}_i, \hat{\bq}_i) = G_i(\hat{\bx}_i, \hat{\bq}_i).
    \end{equation}
    On the other hand, \cref{eq:BMM-i} implies $G_i( \bx^{t_j+1}_i, \bq^{t_j}_i ) \leq G_i( \bxi, \bq^{t_j}_i ) $ for every $\bxi \in \cM_i$, which with $t_j\to \infty$ yields $G_i( \hat{\by}_i, \hat{\bq}_i ) \leq G_i( \bxi, \hat{\bq}_i ) $ for every $\bxi \in \cM_i$. With \cref{eq:tmp-1231243544235}, we see that $\hat{\bx}_i$ and $\hat{\by}_i$ are both global minimizers of $G_i(\bxi, \hat{\bq}_i)$ in variable $\bxi$ over $\cM_i$. It is now only possible that $\hat{\bx}_i=\hat{\by}_i$ by our assumption of a unique global minimizer of $G_i$; therefore $\hat{\bq}_i = \hat{\bx}$ for every $i$. We can now conclude for  every $i=1,\dots,b$ that
    \begin{equation*}
        \begin{split}
            &\ G_i( \hat{\bx}_i, \hat{\bx} ) \leq G_i( \bxi, \hat{\bx} ), \ \  \ \  \ \  \ \  \ \  \ \  \ \  \ \  \ \ \ \forall \bxi \in \cM_i \\ 
        \Rightarrow&\ \ \langle \nabla_{\bx_i} G_i(\bx_i, \hat{\bx})|_{\bx_i= \hat{\bx}_i }, \bv_i \rangle \geq 0,\ \ \  \forall \bv_i\in T_{\cM_i}(\hat{\bx}_i)\\ 
        \Leftrightarrow&\  \ \langle \nabla_i F(\hat{\bx}), \bv_i \rangle \geq 0,\ \ \ \ \  \ \  \ \  \ \  \ \  \ \  \ \  \ \ \ \  \ \ \forall \bv_i\in T_{\cM_i}(\hat{\bx}_i) 
        \end{split}
    \end{equation*}
    The last equivalence is due to the definition of block-wise majorizers (\cref{definition:BM}). Hence $\hat{\bx}$ is a coordinate-wise stationary point (\cref{def:cws}) of \cref{eq:obj} and thus a stationary point (\cref{prop:cwm->cws->s}).
\end{proof}

\subsection{Proof of \cref{theorem:RGD}}
\begin{proof}
Note that the gradient of $F$ is block-$i$ Lipschitz continuous for every $i$, so $F$ has Lipschitz smooth (globally), and therefore $F$ is continuously differentiable. Furthermore, since $\cM_i$ is compact, the conditions of \cref{lemma:block-Lipschitz} are fulfilled. we have for every $i=1,\dots,b$ that
    \begin{equation}\label{eq:descent-rgrad-s}
        \begin{split}
            F(\bx^{t+1}_{1:i}, \bx^t_{i+1:b})
        &=F(\bx^{t+1}_{1:i-1}, \retr_{\bx_i^{t}}(\bs^t_i), \bx^t_{i+1:b}) \\
        &\overset{ \textnormal{(i)} }{\leq} F(\bx^{t+1}_{1:i-1}, \bx^t_{i:b}) + \langle \rgrad_i F(\bx^{t+1}_{1:i-1}, \bx^t_{i:b}), \bs^t_i \rangle + \frac{\rL_i}{2} \cdot \| \bs^t_i \|^2 \\
        &\overset{ \textnormal{(ii)} }{=} F(\bx^{t+1}_{1:i-1}, \bx^t_{i:b}) - \frac{\rL_i}{2} \cdot \big\| \bs^t_i \big\|^2 \\
        &\overset{ \textnormal{(iii)} }{=}  F(\bx^{t+1}_{1:i-1}, \bx^t_{i:b}) - \frac{1}{2\rL_i} \cdot \big\| \rgrad_i F(\bx^{t+1}_{1:i-1}, \bx^t_{i:b}) \big\|^2
        \end{split}
    \end{equation}
Here, (i) follows from \cref{lemma:block-Lipschitz}, and (ii) and (iii) follow from \cref{algo:BEM} with stepsize $\lambda_i^t:=1/\rL_i$, which defines $\bs^t_i:=-\lambda_i^t\rgrad_i F(\bx^{t+1}_{1:i-1}, \bx^t_{i:b})$. In particular, \cref{eq:descent-rgrad-s} implies $F(\bx^{t+1}_{1:i}, \bx^t_{i+1:b}) \leq F(\bx^{t+1}_{1:i-1}, \bx^t_{i:b})$ for every $i=1,\dots,b$ and $t$, and therefore
\begin{equation}\label{eq:diff<diff}
    \begin{cases}
        \frac{\rL_i}{2} \cdot \big\| \bs^t_i \big\|^2  \\ 
        \frac{1}{2\rL_i} \cdot \big\| \rgrad_i F(\bx^{t+1}_{1:i-1}, \bx^t_{i:b}) \big\|^2 
    \end{cases} \leq  F(\bx^{t+1}_{1:i-1}, \bx^t_{i:b}) - F(\bx^{t+1}_{1:i}, \bx^t_{i+1:b}) \leq F(\bx^t) - F(\bx^{t+1}).
\end{equation}

On the other hand, note that we have the following basic equality:
\begin{equation}\label{eq:telegram-equality}
    \nabla_i F(\bx^{t}) = \nabla_i F(\bx^{t+1}_{1:i-1}, \bx^t_{i:b}) + \sum_{j=1}^i \Big( \nabla_i F(\bx^{t+1}_{1:j-1}, \bx^t_{j:b}) - \nabla_i F(\bx^{t+1}_{1:j}, \bx^t_{j+1:b})  \Big)
\end{equation}
Let $\proj_S(\cdot)$ be the operator that projects a point onto set $S$. Then it holds that
    \begin{equation}\label{eq:bound-rgradi}
        \begin{split}
            \big\| \rgrad_i F(\bx^{t}) \big\| &=  \big\| \proj_{T_{\bx_i}\cM_i } \nabla_i F(\bx^{t}) \big\|  \\ 
            &\overset{\textnormal{(i)} }{ \leq } \big\| \rgrad_i F(\bx^{t+1}_{1:i-1}, \bx^t_{i:b}) \big\| + \sum_{j=1}^i \big\| \nabla_i F(\bx^{t+1}_{1:j-1}, \bx^t_{j:b}) - \nabla_i F(\bx^{t+1}_{1:j}, \bx^t_{j+1:b})  \big\| \\ 
            &\overset{\textnormal{(ii)} }{ \leq } \big\| \rgrad_i F(\bx^{t+1}_{1:i-1}, \bx^t_{i:b}) \big\| + \sum_{j=1}^i L_j \cdot \| \bx^t_j - \bx^{t+1}_j \| \\ 
            &= \big\| \rgrad_i F(\bx^{t+1}_{1:i-1}, \bx^t_{i:b}) \big\| + \sum_{j=1}^i L_j \cdot \| \bx^t_j - \retr_{\bx^t_j}(\bs^t_j) \| \\ 
            &\overset{\textnormal{(iii)} }{ \leq }  \big\| \rgrad_i F(\bx^{t+1}_{1:i-1}, \bx^t_{i:b}) \big\| +  \alpha^\textnormal{max}_b L^\textnormal{max}_b \cdot \sum_{j=1}^i \| \bs^t_j \| \\ 
            &\overset{\textnormal{(iv) }  }{ \leq } \sqrt{ 2\rL_i \Big( F(\bx^t) - F(\bx^{t+1}) \Big) } + \alpha^\textnormal{max}_b L^\textnormal{max}_b \cdot \sqrt{ \frac{2i}{\rL^\textnormal{min}_b }  \Big( F(\bx^t) - F(\bx^{t+1}) \Big) } \\ 
            &\overset{ \textnormal{(v)} }{ \leq }  C_b \cdot \sqrt{F(\bx^t) - F(\bx^{t+1}) }
        \end{split}
    \end{equation}
In the above, (i) follows from \cref{eq:telegram-equality}, the triangle inequality, and the fact that the projection $\proj_{T_{\bx_i}\cM_i }(\cdot)$ never increases the norm of its input, (ii) follows from the assumption that $F$ is block-$i$ Lipschitz smooth, (iii) follows from \cref{lemma:block-Lipschitz} and the definitions of $\alpha^\textnormal{max}_b$ and $L^\textnormal{max}_b$, (iv) follows from \cref{eq:diff<diff} and the definition of $\rL^\textnormal{min}_b $, (v) follows from the definition of $C_b$ in \cref{eq:C} and the fact $i\leq b$. Then we have
\begin{equation*}
    \big\| \rgrad F(\bx^{t}) \big\|^2 = \sum_{i=1}^b \big\| \rgrad_i F(\bx^{t}) \big\|^2 \leq  b\cdot C^2_b \cdot \big(F(\bx^t) - F(\bx^{t+1}) \big).
\end{equation*}

Summing the above over $t=0,\dots,T$, we get 
\begin{equation}
    \sum_{t=0}^T \big\| \rgrad F(\bx^{t}) \big\|^2 \leq b\cdot C^2_b \cdot \big(F(\bx^0) - F(\bx^{T+1}) \big),
\end{equation}
from which the conclusion \cref{eq:BRGD-sublinear} follows. 
\end{proof}

\subsection{Proof of \cref{theorem:BEM-RGD}}
\begin{proof}
    The proof follows from that of \cref{theorem:RGD} with the following modifications:
    \begin{itemize}
        \item Note that $F$ is block-wise Lipschitz smooth for the first $b-1$ blocks, and that exact minimization yields a smaller or same objective value than Riemannian gradient descent, therefore \cref{eq:descent-rgrad-s} holds for the iterates of \cref{algo:BEM-RGD} for every $i=1,\dots,b-1$. 
        \item Then, following the proof of \cref{theorem:RGD}, we get a counterpart of \cref{eq:bound-rgradi}, that is
        \begin{equation}\label{eq:bound-rgradi-b-1}
            \big\| \rgrad_i F(\bx^{t}) \big\| \leq C_{b-1} \cdot \sqrt{F(\bx^t) - F(\bx^{t+1})},\ \ \ \  \forall i=1,\dots,b-1.
        \end{equation}
        \item Finally, since \cref{algo:BEM-RGD} performs exact minimization for block $b$, we have $\big\| \rgrad_b F(\bx^{t}) \big\|=0$ for every $t$ (note how $\bx^0_b$ is initialized). and therefore
        \begin{equation}
            \big\| \rgrad F(\bx^{t}) \big\|^2 = \sum_{i=1}^{b-1} \big\| \rgrad_i F(\bx^{t}) \big\|^2 \leq  (b-1)\cdot C^2_{b-1} \cdot \big(F(\bx^t) - F(\bx^{t+1}) \big).
        \end{equation}
        Summing the above over $t=0,\dots,T$ finishes the proof. \qedhere
    \end{itemize}
\end{proof}

\section{Auxiliary Results}\label{section:proof-auxiliary}
\subsection{Proof of \cref{lemma:block-Lipschitz}}
\begin{proof}
    
    
    Note that \cref{eq:alpha_i} was proved in Appendix B of \cite{Boumal-IMA-J-NA2019}, where it was also shown that
    \begin{equation}
        \| \retr_{\bxi}(\bs_i) - \bxi - \bs_i \| \leq \beta_i \cdot \| \bs_i \|^2
    \end{equation}
    for some $\beta_i>0$. Next, we prove \cref{eq:rL_i}. Since $F$ is block-$i$ Lipschitz smooth with constant $L_i$, we have (see, e.g., \cite[Lemma 1.2.3]{Nesterov-2018})
    \begin{equation*}
        F(\bx_{1:i-1},\bzeta,\bx_{i+1:b}) \leq F(\bx_{1:i-1},\bxi,\bx_{i+1:b}) + \langle \nabla_i F(\bx_{1:i-1},\bxi,\bx_{i+1:b}), \bzeta - \bxi \rangle + \frac{L_i}{2} \| \bzeta - \bxi \|^2
    \end{equation*}
    for every $[\bx_{1:i-1}; \bxi; \bx_{i+1:b}]\in\cM$ and every $\bzeta\in\cM_i$ (see, e.g., \cite[Lemma 1.2.3]{Nesterov-2018}). This in particular holds true for $\bzeta=\retr_{\bxi}(\bs_i)$ for every $\bs_i\in T_{\bx_i}\cM_i$. On the other hand, we have
    \begin{equation*}
        \begin{split}
            &\ \langle \nabla_i F(\bx_{1:i-1},\bxi,\bx_{i+1:b}), \retr_{\bxi}(\bs_i) - \bxi \rangle \\ 
            =&\ \langle \nabla_i F(\bx_{1:i-1},\bxi,\bx_{i+1:b}), \retr_{\bxi}(\bs_i) - \bxi + \bs_i - \bs_i \rangle \\ 
            =&\  \langle \rgrad_i F(\bx_{1:i-1},\bxi,\bx_{i+1:b}), \bs_i \rangle +  \langle \nabla_i F(\bx_{1:i-1},\bxi,\bx_{i+1:b}), \retr_{\bxi}(\bs_i) - \bxi - \bs_i \rangle 
        \end{split}
    \end{equation*}
    Combining the above yields
    \begin{equation*}
        \begin{split}
            &\ F(\bx_{1:i-1},\retr_{\bxi}(\bs_i),\bx_{i+1:b}) - F(\bx_{1:i-1},\bxi,\bx_{i+1:b}) - \langle \rgrad_i F(\bx_{1:i-1},\bxi,\bx_{i+1:b}),  \bs_i \rangle \\ 
            \leq&\   \| \nabla_i F(\bx_{1:i-1},\bxi,\bx_{i+1:b})\| \cdot \| \retr_{\bxi}(\bs_i) - \bxi - \bs_i \|  +  \frac{L_i}{2} \| \retr_{\bxi}(\bs_i)- \bxi \|^2  \\ 
            \leq &\ \| \nabla_i F(\bx_{1:i-1},\bxi,\bx_{i+1:b})\| \cdot \beta_i \cdot \| \bs_i \|^2 + \frac{L_i}{2} \cdot \alpha_i^2 \| \bs_i \|^2 \\ 
            = & \ \big(  \| \nabla_i F(\bx_{1:i-1},\bxi,\bx_{i+1:b})\| \cdot \beta_i + \frac{L_i}{2} \cdot \alpha_i^2 \big) \cdot  \| \bs_i \|^2
        \end{split}
    \end{equation*}
    Since $F$ is continuously differentiable and $\cM_i$'s are compact, the gradient $\| \nabla F(\bx_{1:i-1},\bxi,\bx_{i+1:b})\| $ is bounded above by some finite constant value. Therefore $\| \nabla_i F(\bx_{1:i-1},\bxi,\bx_{i+1:b})\|$ is also bounded above, at least by the same constant. The proof is now complete.
\end{proof}

\subsection{\cref{lemma:tangent-cone-inclusion} and Its Proof}
\begin{lemma}\label{lemma:tangent-cone-inclusion}
    Assume each $\cM_i$ is a closed subset of $\bbR^{n_i}$. With $\bx_i\in\cM_i$ it holds that
    \begin{equation*}
        T_\cM(\bx_{}) \subset T_{\cM_1} (\bx_1)\times \cdots \times T_{\cM_b} (\bx_b).
    \end{equation*}
\end{lemma}
\begin{proof}[Proof of \cref{lemma:tangent-cone-inclusion}]
    This result appears in \cite[Prop. 6.4.1]{Rockafellar-2009}; its proof is omitted as the authors wrote it is ``\textit{an elementary consequence}'' of \cref{def:tangent-cone}. Here we give a proof for completeness. Suppose $\bv=\bv_{1:b}\in T_{\cM}(\hat{\bx})$, then there exist some sequence $\{\ba_j\}_j\subset \cM$ and $\{\lambda_j\}_j\subset (0,\infty)$, such that 
    \begin{equation*}
     \lim_{j\to\infty} \ba_j = \hat{\bx},\ \ \ \ \lim_{j\to\infty} \lambda_j=0 ,\ \ \ \ \lim_{j\to\infty} \frac{\ba_j - \hat{\bx}}{\lambda_j} = \bv.
    \end{equation*}
    Write $\ba_j:=[\ba_{j1};\dots;\ba_{jb}]$. The above implies coordinate-wise limits, i.e., we have ($\forall i=1,\dots,b$)
    \begin{equation*}
        \lim_{j\to\infty} \ba_{ji} = \hat{\bx}_i,\ \ \ \ \lim_{j\to\infty} \lambda_j=0 ,\ \ \ \ \lim_{j\to\infty} \frac{\ba_{ji} - \hat{\bx}}{\lambda_j} = \bv_i.
    \end{equation*}
    Thus $\bv_i\in T_{\cM_i}(\hat{\bx}_i)$ for every $i=1,\dots,b$ and therefore $\bv\in T_{\cM_1} (\hat{\bx}_1) \times \cdots \times T_{\cM_b} (\hat{\bx}_b)$.
\end{proof}
\subsection{\cref{prop:cwm->cws->s} and Its Proof}
\begin{definition}\label{def:cws}
We call $\hat{\bx}\in\cM$ a \textit{coordinate-wise stationary point} of \cref{eq:obj} if
\begin{equation}\label{eq:cws}
    \langle \nabla_i F(\hat{\bx}), \bv_i \rangle \geq 0,\ \ \  \forall \bv_i\in T_{\cM_i}(\hat{\bx}_i),\ \forall i=1,\dots,b. 
\end{equation}
\end{definition}
\begin{prop}\label{prop:cwm->cws->s}
Suppose \cref{assumption:basic} holds and $\hat{\bx}$ is a coordinate-wise minimizer of \cref{eq:obj}. Then it is a coordinate-wise stationary point of \cref{eq:obj}, which implies $\hat{\bx}$ is a stationary point of \cref{eq:obj}.
\end{prop}
\begin{proof}
    Since $\hat{\bx}=\hat{\bx}_{1:b}$ is a coordinate-wise minimizer, it is also coordinate-wise stationary, satisfying \cref{eq:cws}. Summing \cref{eq:cws} up, we get $\langle \nabla F(\hat{\bx}), \bv \rangle \geq 0$ for every $\bv\in T_{\cM_1}(\hat{\bx}_1)\times \cdots \times T_{\cM_b}(\hat{\bx}_b)$. Since $T_{\cM_1}(\hat{\bx}_1)\times \cdots \times T_{\cM_b}(\hat{\bx}_b)$ contains $T_{\cM}(\hat{\bx})$ (cf.  \cref{lemma:tangent-cone-inclusion}), we finished the proof.
\end{proof}

\section{Block Majorization Minimization and \cref{algo:BEM} }\label{section:BMM-rate}
In the main paper, we proved the asymptotic convergence of block majorization minimization (\cref{theorem:BMM}). Here, we prove a convergence rate for the same algorithm; similar rate guarantees have been derived in \cite{Hong-MP2017} for the case where both $F$ and $\cM_i$'s are convex, but we will proceed without any convexity assumptions. 

\begin{assumption}\label{assumption:majorizer-LGi}
    For each $i=1,\dots,b$, $F$ admits a block-$i$ majorizer $G_i$. Moreover, $G_i$ is Lipschitz smooth with constant $L(G_i)>0$, that is $G_i$ satisfies
     \begin{equation}
         \|\nabla G_i(\bxi,\bz) - \nabla G_i(\bzeta,\bz) \| \leq L(G_i) \cdot \| \bxi - \bzeta \|, \ \ \ \forall \bz\in \cM, \ \forall \bxi\in \cM_i, \forall \bzeta\in \cM_i,
     \end{equation}
     where $\nabla G_i(\bxi,\bz)$ denotes the gradient of $G_i$ with respect to its first parameter $\bxi$.
\end{assumption}
We then reach a corollary of \cref{lemma:block-Lipschitz}:
\begin{corollary}\label{corollary:Lipschitz-Gi}
    Let $i\in\{1,\dots,b\}$ be fixed. If \cref{assumption:majorizer-LGi} holds and $\cM_i$ is a compact submanifold of $\bbR^{n_i}$, then there is some constant $\rL(G_i)>0$ such that (for every $\bz\in \cM $ and every $\bs_i\in T_{\bz_i} \cM_i$)
    \begin{equation}\label{eq:rL_G_i}
        G_i( \retr_{\bz_i}(\bs_i), \bz ) \leq G(\bz_i, \bz) + \langle \rgrad G(\bz_i, \bz), \bs_i  \rangle + \frac{\rL(G_i)}{2} \cdot \|\bs_i \|^2.
    \end{equation}
\end{corollary}

We can now state and prove the desired result:
\begin{theorem}\label{theorem:BMM-rate}
    For each $i=1,\dots,b$ assume $\cM_i$ is a compact smooth submanifold of $\bbR^{n_i}$ and $F$ is block-$i$ Lipschitz smooth with constant $L_i$. Suppose \cref{assumption:majorizer-LGi} holds. Let $\rL(G_i)$ be defined in \cref{corollary:Lipschitz-Gi} and $\alpha^\textnormal{max}_b$ and $ L^\textnormal{max}_b$ in \cref{eq:C}. If \cref{algo:BEM} performs updates via \cref{eq:BMM-i}, then its iterates $\{\bx^t\}_{t}$ satisfy
 \begin{equation*}
     \small{\min_{t=0,\dots,T} \big\| \rgrad F(\bx^{t}) \big\| \leq \sqrt{b}\cdot \Bigg( \sqrt{2  \max_{i=1,\dots,b}\rL(G_i) } + \sqrt{ \frac{2b \cdot (\alpha^\textnormal{max}_b L^\textnormal{max}_b)^2 }{ \min_{i=1,\dots,b}\rL(G_i) }  } \ \Bigg) \cdot \sqrt{ \frac{F(\bx^0) - F(\bx^{T+1}) }{T+1} }.}
 \end{equation*}
\end{theorem}
\begin{remark}
    \cref{theorem:BMM-rate} assumed that $F$ is block-$i$ Lipschitz smooth for every $i$. Replacing it with a smoothness assumption on $G_i$ on the second parameter (see, e.g., \cite[Eq. 2.9]{Hong-MP2017}) would lead to a similar convergence guarantee. Furthermore, this blockwise Lipschitz smoothness assumption on $F$ can be removed if we only have a single block ($b=1$).
\end{remark}
\begin{proof}
    Recall the definition $\bq_i^t:=[\bx^{t+1}_{1:i-1}; \bx^t_{i:b}]$ in \cref{eq:BMM-i}. Then for every $i=1,\dots,b$ we have
    \begin{equation}
        \begin{split}
            F(\bx^{t+1}) - F(\bx^{t}) &\overset{\textnormal{(i)} }{\leq} G_i(\bx_{i}^{t+1}, \bq_i^t) - G_i(\bx_{i}^{t}, \bq_i^t) \\ 
            &\overset{\textnormal{(ii)} }{\leq} - \frac{1}{2 \rL(G_i)} \cdot \big\| \rgrad G(\bx_i^t, \bq_i^t) \big\|^2 \\ 
            &\overset{\textnormal{(iii)} }{\leq} - \frac{1}{2 \rL(G_i)} \cdot \big\| \rgrad_i F(\bx_{1:i-1}^{t+1},\bx_{i:b}^t) \big\|^2 
        \end{split}
    \end{equation}
    In the above, (i) follows from the definition of block majorizers $G_i$'s, see also \cref{eq:non-increasing-G}; (ii) follows from the minimization property \cref{eq:BMM-i} and \cref{corollary:Lipschitz-Gi}, where \cref{eq:rL_G_i} holds with $\bs_i= - \rgrad G(\bx_i^t, \bq_i^t) / \rL(G_i)$; (iii) holds because $\nabla G(\bx_i^t, \bq_i^t) = \nabla_i F(\bq_i^t) = \nabla_i F(\bx_{1:i-1}^{t+1},\bx_{i:b}^t)$ and therefore the corresponding Riemannian gradients are also equal. We can then see that 
    \begin{equation}\label{eq:diff<diff-G}
        \begin{cases}
            \frac{\rL(G_i)}{2} \cdot \big\| \bs^t_i \big\|^2  \\ 
            \frac{1}{2\rL(G_i)} \cdot \big\| \rgrad_i F(\bx^{t+1}_{1:i-1}, \bx^t_{i:b}) \big\|^2 
        \end{cases} \leq F(\bx^t) - F(\bx^{t+1}),
    \end{equation}
    which is identical to \cref{eq:diff<diff} except that we have $\rL(G_i)$ here, instead of $\rL_i$. Therefore, except for that difference, the rest of the proof is identical to the corresponding part of \cref{theorem:RGD} (note that \cref{eq:bound-rgradi} used the assumption that $F$ is block-$i$ Lipschitz smooth for every $i=1,\dots,b$).
\end{proof}

\section{GPCA: A General Version}\label{subsection:GPCA-b}
The main paper discussed the GPCA problem with two subspaces $\cS_1$ and $\cS_2$. Here we present it  for $b$ subspaces $\cS_1,\dots,\cS_b$. Note that the exposition here has no fundamental difference from the main paper, except that the mathematical notations in the general case are slightly denser.

\myparagraph{GPCA} A point $\bp$ belongs to $\cup_{i=1}^b \cS_i$ if and only if $\prod_{i=1}^b \| \bp_j^\top \bA_i^* \| =0$. Then, we can formulate the problem as follows (see \cite[Eq. 3.87]{Vidal-2003thesis}):
\begin{equation}\label{eq:GPCA-LS-b}
    \min_{ \{\bA_i\}_{i=1}^b } \sum_{j=1}^m \prod_{i=1}^b \| \bp_j^\top \bA_i \|^2 \ \ \ \ \ \textnormal{s.t.} \ \ \ \   \bA_i^\top \bA_i  = \bI_{c_i}, \ \forall k=1,\dots, b
\end{equation}
At the time of \cite{Vidal-2003thesis}, solving \cref{eq:GPCA-LS-b} seemed to be considered difficult; no algorithm has been proposed for and applied to  \cref{eq:GPCA-LS-b} so far, to our knowledge. 

With initialization $\bA_1^0,\dots,\bA_b^0$, the block minimization step \cref{eq:BEM-i} turns out to be
\begin{equation}\label{eq:BEM-i-GPCA-b}
    \bA_i^{t+1} \in \argmin_{\bxi} \sum_{j=1}^m w_{i,j}^t\cdot  \| \bp_j^\top \bxi \|^2  \ \ \ \ \ \textnormal{s.t.} \ \ \ \   \bxi^\top \bxi  = \bI_{c_i}
\end{equation}
where the weight $w_{i,j}^t$ of \cref{eq:BEM-i-GPCA-b} is defined to match the objectives of \cref{eq:GPCA-LS-b} with the current iterates, i.e.,
\begin{equation}\label{eq:GPCA-LS-weight-b}
    w_{i,j}^t:= \begin{cases}
    \prod_{i'>i}\| \bp_j^\top \bA_{i'}^t \|^2 & i = 1 \\
    \prod_{i' < i} \| \bp_j^\top \bA_{i'}^{t+1} \|^2\cdot \prod_{i' > i} \| \bp_j^\top \bA_{i'}^{t} \|^2 & i \neq 1
    \end{cases} 
\end{equation}
Similarly to \cref{eq:BEM-i-GPCA}, we can solve \cref{eq:BEM-i-GPCA-b} by eigenvalue decomposition.

\myparagraph{GPCA with Huber-Style Loss} A direct extension of \cref{eq:GPCA-Huber} to multiple subspaces is
\begin{equation}\label{eq:GPCA-Huber-b}
    \min_{ \{\bA_i\}_{i=1}^b } \sum_{j=1}^m \prod_{i=1}^b h(\| \bp_j^\top \bA_i \|) \ \ \ \ \ \textnormal{s.t.} \ \ \ \   \bA_i^\top \bA_i  = \bI_{c_i}, \ \forall k=1,\dots, b
\end{equation}
Then, \cref{algo:BEM} applied to \cref{eq:GPCA-Huber-b} amounts to solving
\begin{equation}\label{eq:BEM-i-GPCA-Huber-b}
    \bA_i^{t+1} \in \argmin_{\bxi} \sum_{j=1}^m v_{i,j}^t\cdot  h(\| \bp_j^\top \bxi \|)  \ \ \ \ \ \textnormal{s.t.} \ \ \ \   \bxi^\top \bxi  = \bI_{c_i}
\end{equation}
for every $i$ and $t$. Similarly to $w_{i,j}^t$ of \cref{eq:GPCA-LS-weight-b}, the weight $v_{i,j}^t$ in \cref{eq:BEM-i-GPCA-Huber-b} should be defined to match the objective of \cref{eq:GPCA-Huber-b}. More specifically, $v_{i,j}^t$ is defined as
\begin{equation*}
    v_{i,j}^t:= \begin{cases}
    \prod_{i'>i} h\big( \|\bp_j^\top \bA_{i'}^t \| \big) & i = 1 \\
    \prod_{i' < i} h\big( \| \bp_j^\top \bA_{i'}^{t+1} \| \big) \cdot h\big( \bp_j^\top \bA_{i'}^t \| \big) & i \neq 1
    \end{cases} 
\end{equation*}

\end{document}